\documentclass[twoside]{article}

\usepackage{a4wide}
\usepackage[T1]{fontenc}	
\usepackage[latin1]{inputenc}	
\usepackage{calc}		

\usepackage{xcolor}
\usepackage[colorlinks=true,citecolor=red,linkcolor=blue,breaklinks=true,bookmarksdepth=1]{hyperref}
\usepackage{xparse}
\newcommand{\answerI}[1]{ #1}

\ExplSyntaxOn
\NewDocumentCommand { \xifnum } { }
    {
        \fp_compare:nTF
    }
\ExplSyntaxOff

\usepackage{amsmath}	
\usepackage{amssymb}	
\usepackage{amsfonts}	
\usepackage{amsthm}	

\usepackage{array}
\usepackage{curves}
\usepackage{epsfig}
\usepackage{graphics}
\usepackage{mathrsfs}

\newdimen\argwidth
\newcommand{\Jump}[1]{%
\setbox0=\hbox{$#1$}\argwidth=\wd0
\setbox0=\hbox{$\left[\box0\right]$}\advance\argwidth by -\wd0
\left[\kern.3\argwidth\box0\kern.3\argwidth\right]}

\newcommand{\bkperp}{\ensuremath{\mathbf{k^{\perp}}}}
\newcommand{\bb}{\ensuremath{\mathbf{b}}}
\newcommand{\bB}{\ensuremath{\mathbf{B}}}
\newcommand{\bn}{\ensuremath{\mathbf{n}}}
\newcommand{\bu}{\ensuremath{\mathbf{u}}}
\newcommand{\be}{\ensuremath{\mathbf{e}}}
\newcommand{\bv}{\ensuremath{\mathbf{v}}}

\newcommand{\bk}{\ensuremath{\mathbf{k}}}

\newcommand{\bC}{\ensuremath{\mathbf{C}}}

\newcommand{\cP}{\ensuremath{\mathcal{P}}}
\newcommand{\cF}{\ensuremath{\mathcal{F}}}
\newcommand{\cC}{\ensuremath{\mathcal{C}}}

\newcommand{\degx}{\ensuremath{\dd^{\circ} _x}}
\newcommand{\degy}{\ensuremath{\dd^{\circ} _y}}

\newcommand{\R}{\ensuremath{\mathbb{R}}}
\newcommand{\T}{\ensuremath{\mathbb{T}}}

\usepackage{bbm}

\newcommand{\PP}{\ensuremath{\mathbb{P}}}
\newcommand{\QQ}{\ensuremath{\mathbb{Q}}}
\newcommand{\KK}{\ensuremath{\mathbb{K}}}
\newcommand{\bKK}{\ensuremath{\pmb{\mathbb{K}}}}
\newcommand{\dPP}{\ensuremath{\mathrm{d}\mathbb{P}}}
\newcommand{\bdPP}{\ensuremath{\pmb{\mathrm{d}\mathbb{P}}}}
\newcommand{\dQQ}{\ensuremath{\mathrm{d}\mathbb{Q}}}
\newcommand{\bdQQ}{\ensuremath{\pmb{\mathrm{d}\mathbb{Q}}}}
\newcommand{\bdQQkCurl}{\ensuremath{\widehat{\bdQQ}_k ^{\mathrm{curl}} }}
\newcommand{\bdQQkDiv}{\ensuremath{\widehat{\bdQQ}_k ^{\mathrm{div}} }}
\newcommand{\dQQkHat}{\ensuremath{\widehat{\dd \QQ}_{k-1} }}
\newcommand{\RT}{\ensuremath{\pmb{\mathbb{RT}}}}
\newcommand{\dRT}{\ensuremath{\pmb{\mathrm{d}\mathbb{RT}}}}
\newcommand{\bN}{\ensuremath{\pmb{\mathbb{N}}}}
\newcommand{\dN}{\ensuremath{\pmb{\mathrm{d}\mathbb{N}}}}

\DeclareMathOperator{\curlS}{\nablaperp \cdot}
\DeclareMathOperator{\Vect}{Vec}
\DeclareMathOperator{\rank}{rank}
\DeclareMathOperator{\Range}{Range}

\DeclareMathOperator{\Tr}{Tr}
\DeclareMathOperator{\bTr}{\mathbf{Tr}}
\newcommand{\divx}{\ensuremath{\nabla \cdot}}
\newcommand{\ddivx}{\ensuremath{{\nabla_{\mathscr{D}'} \cdot}}}
\newcommand{\nablaperp}{\ensuremath{{\nabla^{\perp}}}}
\newcommand{\dnablaperp}{\ensuremath{\nabla^{\perp} _{\mathscr{D}'}}}

\newcommand{\dcurlS}{\ensuremath{\nabla^{\perp} _{\mathscr{D}'} \cdot}}
\newcommand{\gradx}{\ensuremath{\nabla}}

\newcounter{propc}
\newtheorem{prop}[propc]{Proposition}

\newcommand{\Sum}[2]{\ensuremath{\textstyle{\sum\limits_{#1}^{#2}}}}
\newcommand{\Int}[2]{\ensuremath{\mathchoice%
{{\displaystyle\int_{#1}^{#2}}}
{{\displaystyle\int_{#1}^{#2}}}
{\int_{#1}^{#2}}
{\int_{#1}^{#2}}
}}

\newcommand{\dd}{\ensuremath{{\rm d}}}

\newcommand{\scalar}[2]{\ensuremath{\langle #1 | #2 \rangle}}

\newcommand{\diag}[1]{\stackrel{\displaystyle{#1}}{\xrightarrow{\qquad\qquad}}}

\usepackage{tikz-cd}
\usetikzlibrary{positioning, shapes.geometric}
\tikzset{>={Stealth[width=1.5mm, length=1mm]}}

\usepackage{ifthen}

\def\LagrangeQuad#1#2#3#4{
  \node[rectangle,draw, minimum size=3.6cm,thick] (r) at (#1,#2) {};
  \foreach \x in {0,1,...,#4}{
    \foreach \y in {0,1,...,#4}{
      \node[color=blue] at (#1-0.5*#3 +#3/#4*\x,#2-0.5*#3 +#3/#4*\y) {\LARGE $\bullet$};
    }
  }
}

\def\dFaceCellsQuad#1#2#3#4{
  \node[rectangle,draw, minimum size=3.6cm,thick] (r) at (#1,#2) {};
  \foreach \x in {0,1,...,#4}{
    \node[color=red] at (#1-0.5*#3 +#3/#4*\x,#2-0.5*4) {\LARGE $\bullet$};
    \node[color=red] at (#1-0.5*#3 +#3/#4*\x,#2+0.5*4) {\LARGE $\bullet$};
    \node[color=red] at (#1-0.5*4,#2-0.5*#3 +#3/#4*\x) {\LARGE $\bullet$};
    \node[color=red] at (#1+0.5*4,#2-0.5*#3 +#3/#4*\x) {\LARGE $\bullet$};
    \foreach \y in {0,1,...,#4}{
      \node[color=blue] at (#1-0.5*#3 +#3/#4*\x,#2-0.5*#3 +#3/#4*\y) {\LARGE $\bullet$};
    }
  }
}

\def\dRaviartThomasQuad#1#2#3#4{
  \node[rectangle,draw, minimum size=3.6cm,thick] (r) at (#1,#2) {};
  \pgfmathsetmacro\npu{1+#4};
  \foreach \x in {0,1,...,#4}{
    \foreach \y in {0,1,...,\npu}{
      \draw[color=blue,<->,very thick] (-0.2+#1-0.5*#3 +#3/#4*\x,#2-0.5*#3 +#3/\npu*\y) -- (0.2+#1-0.5*#3 +#3/#4*\x,#2-0.5*#3 +#3/\npu*\y);
    }
  }

  \foreach \x in {0,1,...,\npu}{
    \foreach \y in {0,1,...,#4}{
      \draw[color=blue,<->,very thick] (#1-0.5*#3 +#3/\npu*\x,-0.2+#2-0.5*#3 +#3/#4*\y) -- (#1-0.5*#3 +#3/\npu*\x,0.2+#2-0.5*#3 +#3/#4*\y);
    }
  }
}

\def\dNedelecQuad#1#2#3#4{
  \node[rectangle,draw, minimum size=3.6cm,thick] (r) at (#1,#2) {};
  \pgfmathsetmacro\npu{1+#4};
  \foreach \x in {0,1,...,#4}{
    \foreach \y in {0,1,...,\npu}{
      \draw[color=blue,<->,very thick] (#1-0.5*#3 +#3/#4*\x,-0.2+#2-0.5*#3 +#3/\npu*\y) -- (#1-0.5*#3 +#3/#4*\x,0.2+#2-0.5*#3 +#3/\npu*\y);
    }
  }

  \foreach \x in {0,1,...,\npu}{
    \foreach \y in {0,1,...,#4}{
      \draw[color=blue,<->,very thick] (-0.2+#1-0.5*#3 +#3/\npu*\x,#2-0.5*#3 +#3/#4*\y) -- (0.2+#1-0.5*#3 +#3/\npu*\x,#2-0.5*#3 +#3/#4*\y);
    }
  }
}

\author{Vincent Perrier \\
Team Cagire, INRIA Bordeaux Sud-Ouest. \\
Laboratoire de Math\'ematiques et de leurs applications \\
B\^atiment IPRA, 
Universit\'e de Pau et des Pays de l'Adour, \\ 
Avenue de l'Universit\'e, 
64\,013 Pau Cedex}

\title{discrete de-Rham complex involving a discontinuous
  finite element space for velocities: the case of periodic straight
  triangular and Cartesian meshes}

\begin{document}

\maketitle

\abstract{
  The aim of this article is to derive discontinuous finite elements
  vector spaces which can be put in a discrete de-Rham complex for which
  \answerI{the matching between the continuous and discrete cohomology spaces
  can be proven for periodic meshes}.

  
  Then the triangular case is addressed, for which we prove that
  \answerI{the same property} holds for the classical discontinuous finite
  element space for vectors.

  On Cartesian meshes,
  this result does not hold for the classical discontinuous
  finite element space for vectors. We then show how to use the de-Rham
  complex found for triangular meshes for enriching the finite element space on
  Cartesian meshes in order to recover a de-Rham complex, on which
  the same \answerI{property} is proven.}

\tableofcontents

\section{Introduction}

In this article, we are interested in the de-Rham complex on a two dimensional
space $\Omega$:
\begin{equation}
  \label{eq:deRhamDifferentialForms}
  \Lambda^0 ( \Omega )
  \diag{\dd}
  \Lambda^1 (  \Omega )
  \diag{\dd}
  \Lambda^2 (  \Omega ), 
\end{equation}
where $\Lambda^k ( \Omega)$ is the set of $k$-differential forms,
and $\dd$ is the exterior derivative. For the sake of simplicity,
we suppose that $\Omega$ is the two dimensional torus $\T^2$.

In the context of partial differential equations, it is usually
convenient to translate the multilinear forms of
\eqref{eq:deRhamDifferentialForms} in terms of \emph{proxies}. This is achieved
by choosing a basis of $\Lambda^{1} (\Omega)$. If $(\be_1, \be_2)$ is an
orthogonal basis, its dual basis is denoted by $(\dd x^1, \dd x^2)$, and
this leads usually to the two following situations
\begin{itemize}
\item When the basis $(\dd x^1, \dd x ^2)$ is used for $\Lambda^1$, and
  $\dd x^1 \wedge \dd x^2$ is used for $\Lambda^2$,
  the exterior derivative between $\Lambda^0$ and $\Lambda^1$
  gives a gradient operator on the proxies. 
  In this case, the $0$--forms of
  \eqref{eq:deRhamDifferentialForms} maps to the set $A$ of the scalar
  potentials, the $1$--forms to a set of vectors
  $\bB$, and the $2$--forms to a set of scalars \answerI{$C$}.
  The exterior derivative between $A$ and $\bB$ is the gradient, $\nabla$:
  $$
  \begin{array}{rcr@{\quad \longmapsto \quad }l}
    \nabla & \, : \, & A
    & \bB \\
    & & a & \bb = \left(\partial_x a , \partial_y a  \right)^T
  \end{array}
  $$
  whereas the exterior derivative between $\bB$ and $C$ is the
  rotated divergence (or scalar curl, which is obtained by taking the
  $z$ component of the classical three dimensional curl)
  $$
  \begin{array}{rcr@{\quad \longmapsto \quad }l}
    \curlS & \, : \, & \bB
    & C \\
    & & \bb & c = - \partial_y \bb_x + \partial_x \bb_y , 
  \end{array}
  $$
  and the diagram \eqref{eq:deRhamDifferentialForms} can be rewritten in term
  of proxies as
  \begin{equation}
    \label{eq:DiagramGrad}
    A \quad \diag{\nabla} \quad \bB \quad \diag{\curlS } \quad C.
  \end{equation}

\item When the basis $(- \dd x^2, \dd x ^1)$ is used for $\Lambda^1$, and
  $\dd x^1 \wedge \dd x^2$ is used for $\Lambda^2$,
  the exterior derivative between $\Lambda^0$ and $\Lambda^1$
  gives the rotated gradient (which can also be seen as a curl) on the proxies.
  In this case, the $0$--forms of
  \eqref{eq:deRhamDifferentialForms} maps to the set $A$ of scalars 
  (which may be seen as potential vectors by taking the curl of a vector which
  would have only a $z$ component), the $1$--forms to a set of vectors
  $\bB$, and the $2$--forms to a set of scalars $C$. The exterior derivative
  between $A$ and $\bB$ is the rotated gradient, $\nablaperp$:
  $$
  \begin{array}{rcr@{\quad \longmapsto \quad }l}
    \nablaperp & \, : \, & A
    & \bB \\
    & & a & \bb = \left(- \partial_y a , \partial_x a  \right)^T
  \end{array}
  $$
  whereas the exterior derivative between $\bB$ and $C$ is the opposite of the 
  two-dimensional divergence
  $$
  \begin{array}{rcr@{\quad \longmapsto \quad }l}
    \nabla \cdot & \, : \, & \bB
    & C \\
    & & \bb & c = \partial_x \bb_x + \partial_y \bb_y, 
  \end{array}
  $$
  and the diagram \eqref{eq:deRhamDifferentialForms} can be rewritten as
  \begin{equation}
    \label{eq:DiagramCurlMinus}
    A \quad \diag{\nablaperp} \quad \bB \quad \diag{-\nabla \cdot} \quad C,
  \end{equation}
  For the sake of simplicity, we will address rather the diagram
  \begin{equation}
    \label{eq:DiagramCurl}
    A \quad \diag{\nablaperp} \quad \bB \quad \diag{\nabla \cdot} \quad C,
  \end{equation}
\end{itemize}

A natural question that arises when considering diagrams such as
\eqref{eq:DiagramGrad} and \eqref{eq:DiagramCurl} is whether the sequence is
exact, which can be summarised for \eqref{eq:DiagramGrad} as :
\begin{itemize}
\item Is the kernel of $\nabla$ reduced to $0$?
\item Is $\left( \curlS \right)$ full rank?
\item Do we have $\Range ( \nabla ) = \ker \left( \curlS \right)$?
\end{itemize}
In general, the answer of the previous questions is "no", however, it is nearly
"yes" in the sense that the dimension of the vectorial spaces
$\ker \nabla$, $C/\Range( \curlS )$
and $\ker \curlS / \Range ( \nabla )$ (called the
\emph{cohomology spaces}) are finite. Following the
Hodge theory, these dimensions equal
to the zeroth, first and second Betti numbers (denoted by $b_0$, $b_1$ and
$b_2$), which are characteristics
of the topology of the domain $\Omega$. We are interested in this article in
two-dimensional periodic domain, namely a torus, for which we have
$b_0 = b_2 = 1$ and $b_1 = 2$. Also, $\ker \nabla$ and $C/\Range( \curlS )$
match with the uniform functions (which is a one dimensional vector space, and
so is consistent with $b_0 = b_2 =1$),
whereas $\ker \left( \curlS \right) / \Range ( \nabla )$
matches with the uniform vectors
(which is a two dimensional vector space, so is consistent with $b_1 = 2$). 

We are now interested in the discrete counterpart of these properties: once
$\Omega$ is discretised by a mesh, do we have a discrete counterpart of
\eqref{eq:DiagramGrad} and \eqref{eq:DiagramCurl}?
This indeed exists in the conforming finite element context
\cite{arnold2014periodic}. For example, for triangular meshes, the
following discrete version of \eqref{eq:DiagramCurl} involving the
space of continuous finite elements $\PP_k$, the space of Raviart-Thomas
finite elements $\RT_k$ and the space of discontinuous finite elements
$\dd \PP_{k-1}$
\begin{equation}
  \label{eq:DiagramCurlTriangle}
  \PP_{k+1} \quad \diag{\nablaperp} \quad \RT_{k+1} \quad \diag{\nabla \cdot}
  \quad \dd \PP_{k},
\end{equation}
whereas a discrete version of \eqref{eq:DiagramGrad} can also be derived
\begin{equation}
  \label{eq:DiagramGradTriangle}
  \PP_{k+1} \quad \diag{\nabla} \quad \bN_{k+1} \quad \diag{\curlS}
  \quad \dd \PP_{k},
\end{equation}
where $\bN_k$ is the space of two dimensional triangular Nédélec first species
finite elements. Several properties of the discrete diagrams
\eqref{eq:DiagramCurlTriangle} and \eqref{eq:DiagramGradTriangle} are
important (see \cite[Chap. 5.2.2]{arnold2018finite}): 
\begin{itemize}
\item {\bf the approximation property:} this property is ensured if
  the discrete spaces are correctly approximating the continuous spaces. 
\item {\bf the subcomplex property:} this property is a compatibility
  property between the discrete complex
  \eqref{eq:DiagramCurlTriangle} and \eqref{eq:DiagramGradTriangle}, which
  should be a subcomplex respectively of the continuous complexes 
  \eqref{eq:DiagramCurl} and \eqref{eq:DiagramGrad}, \answerI{this means for
    example for the diagram \eqref{eq:DiagramCurlTriangle}, 
    $\nabla^{\perp} \left( \PP_{k+1}\right) \subset \RT_{k+1}$ and
  $\nabla \cdot \left( \RT_{k+1} \right) \subset \dd \PP_k$ are ensured.}
\item {\bf the bounded cochain projection property:} this property
  \answerI{ means the existence of projection operators, e.g. between the
    continuous and discrete spaces of \eqref{eq:DiagramCurlTriangle}, that
    commutes with the exterior derivative, and that is bounded. This property
    is usually not easy to address, as the canonical interpolant are
    not bounded, see \cite[Section 5.4]{arnold2006finite}
    for an example of construction of a bounded cochain projection in the
    conformal case.}
\end{itemize}
\answerI{The second and third properties,
  combined with an additional approximation property, induce another
  property, the \emph{isomorphism of cohomology}
  \cite[Theorem 5.1]{arnold2018finite}. Another property, the
  \emph{gap between harmonic forms}
  \cite[Theorem 5.2]{arnold2018finite} controls the gap between the continuous
  and discrete harmonic forms. All these properties depend on the
  definition of the bounded cochain projection, which is not yet defined
  for the spaces we wish to address, and this is why we focus on a
  simplified case, the periodic case, for which the cohomology spaces
  are explicit, and match respectively for their proxies with
  constant scalar, constant vectors, and constant scalars. In this article,
  we would like to address the following proposition
  \begin{prop}
    \label{def:harmonicGap}
    The discrete cohomology space matches exactly
    with the continuous one.
  \end{prop}
  which is expected to hold for periodic meshes. 
}

The finite element exterior calculus has been thoroughly addressed over
the last thirty years, first in the electromagnetism context
\cite{bossavit1988whitney,bossavit1998computational,hiptmair2001discrete,hiptmair2002finite}, and then extended to the
slightly more abstract Hodge Laplacian problem
\cite{arnold2006finite,arnold2010finite}, and led to a quite complete
theory for conforming finite elements on
classical cells (quads, triangles, hexa and
tetrahedra) \cite{arnold2018finite}.
This type of approximation was extended to polytopal
meshes see e.g.
\cite{bonelle2014compatible,bonelle2014analysis,bonelle2015analysis,
  milani2022artificial}
for the "Compatible Discrete Operators" framework
or \cite{di2020hybrid} for the "Hybrid High Order" method for citing few of
these methods. 
For the classical discontinuous Galerkin methods, as far as we know, few work
was considered see however e.g. \cite{hong2022extended} for
recent advances on this topic for the Hodge Laplacian. 

In this article, we are interested in finding a discrete counterpart of
\eqref{eq:DiagramGrad} and \eqref{eq:DiagramCurl} when $\Omega$ is
meshed with a triangular mesh or  with a
Cartesian mesh, and with a particular constraint: we want the space of
vectors $\bB$ to be discretised with 
discontinuous finite elements.
\answerI{This constraint is in fact motivated by recent results
  \cite{guillard2009behavior,jung2024behavior}, which suggests
  that not only the classical discontinuous approximation space on triangles has
  a structure that allows the preservation of curl constraints for
  hyperbolic systems, but also that a similar structure exists on
  quadrangular meshes, see e.g. \cite{JUNG2024113252} for the low order
  quadrangular case.
  Note that currently, the problem of finding curl or divergence
  preserving schemes  is usually addressed with staggered schemes
  \cite{tavelli2017pressure,balsara2001divergence,balsara2004second,
    balsara1999staggered}
  (based on the discrete de-Rham complex of \cite{arnold2018finite}),
  which makes the task of limiting for shocks while remaining
  conservative difficult, because the degrees of freedom are spread
  on all the entities of the mesh, and also because the notion of local
  conservation is hard to define. 
}

\answerI{This article is focused on addressing
  \autoref{def:harmonicGap} for these discontinuous approximation
spaces, but
without addressing  bounded cochain projection
property}, which are complicated to address in the nonconformal case, and
out of the scope of this paper. As the approximation is nonconforming,
the classical differential operators cannot be considered, and discrete
differential operators shall be defined. For the sake of simplicity, we
will consider differential operators matching with the derivation in the
sense of distributions, leading to approximation spaces that are
Cartesian products of approximation spaces on different entities of the mesh,
similar to the Hybrid High Order framework \cite{di2020hybrid}. The
complexes considered are similar to the ones of
\cite{licht2017complexes}, but include a lower number of degrees of freedom. 

The article is organised as follows.
In \autoref{sec:notations}, the notations for the mesh and the finite
element space and the discrete differential operators are given. 
Some enumeration properties of the mesh are also proven in this section.
Then in \autoref{sec:dRT}, we recall the results of
\cite{licht2017complexes} for a choice of vector finite elements inspired
by the conformal
case \eqref{eq:DiagramGradTriangle},\eqref{eq:DiagramCurlTriangle}. 
Then, in \autoref{sec:triangularMesh}, we prove that if $\Omega$ is meshed
with triangles and if $\bB$ is
approximated by the usual discontinuous finite element space, then it
can be put in a discrete diagram similar to \eqref{eq:DiagramGrad} and
\eqref{eq:DiagramCurl} where the space $A$ is approximated by the continuous
finite element space. The \autoref{def:harmonicGap} is proven for this discrete
diagram.  
Then in \autoref{sec:CartesianMesh}, the same problem is
addressed for Cartesian meshes. We first prove that \autoref{def:harmonicGap}
fails for $k=0$ with the classical piecewise constant finite element vector space. Inspired by the diagram that holds on triangles,
we prove that by enriching the classical discontinuous finite element space,
\autoref{def:harmonicGap} can be recovered.
\autoref{sec:Conclusion} is the conclusion. 

\section{Notations}
\label{sec:notations}

\subsection{Mesh notations}

We denote by $\cP$ the set of points of the mesh, by
$\cC$ the set of cells of the mesh, and by $\cF$ the set of the
faces of the mesh.
For a given entity, for example a cell $c$, we denote by
$\cF (c)$ the set of faces neighbouring the cell $c$\answerI{, and by
$\cC(f)$ the set of cells neighbouring the face $f$.}

Each face \answerI{joining points $P$ and $Q$} is supposed to be oriented,
and we denote by
$\bn_f$ the unit normal to the face $f \in \cF$ that is
positive\answerI{, namely such that the angle between the vectors $\bn_f$
  and $\overrightarrow{PQ}$
  is positive. Then the neighbouring cell of this face such that the
  normal $\bn_f$ is inward is the right cell, and the other is the left cell}.
If $\bu$ is a
vector that is discontinuous through the face $f$, then its jump
$\Jump{\cdot}$ is defined as
$$\Jump{\bu \cdot \bn_f} = \bu_R \cdot \bn_f - \bu_L \cdot \bn_f , $$
where $\bu_L$ is the value on the left and $\bu_R$ is the value on the right.

\begin{prop}[Triangular mesh of a torus]
  \label{prop:NElementsTri}
  For a triangular mesh, if $N$ denotes the number of cells, then
  $$
  \left\{
  \begin{array}{r@{\, = \, }l}
    \# \cC & N \\
    \# \cF & \dfrac{3N}{2} \\
    \# \cP & \dfrac{N}{2}. \\
  \end{array}
  \right.
  $$

\end{prop}

\begin{proof}
  We remark that the following sum
  $$
  \Sum{f \in \cF}{}
  \Sum{c \in \cC (f)}{} 1,
  $$
  can be computed in two manners: on one hand, we have two cells per face,
  so this sum is equal to $2 \# \cF$. On the other hand, when doing this
  sum, each cell is visited $3$ times (because each cell has three faces), and
  so the sum is equal to $3 N$. This gives $\# \cF = \dfrac{3N}{2}$.

  The Euler formula states that
  $$\# \cP - \# \cF + \# \cC = 2(1-g),$$
  where $g$ is the genus of the surface. As we are dealing with a
  two-dimensional domain, with periodic boundary conditions,
  this is a torus in three dimensions, so that $g=1$. This leads to
  $$\# \cP = \# \cF - \# \cC = \dfrac{3 N}{2}  - N = \dfrac{N}{2}.$$
\end{proof}

\begin{prop}[Cartesian mesh of a torus]
  \label{prop:NElementsQuad}
  For a Cartesian mesh with periodicity,
  if $N$ denotes the number of cells, then
  $$
  \left\{
  \begin{array}{r@{\, = \, }l}
    \# \cC & N \\
    \# \cF & 2N \\
    \# \cP & N. \\
  \end{array}
  \right.
  $$
\end{prop}

\begin{proof}
  We remark that the following sum
  $$
  \Sum{f \in \cF}{}
  \Sum{c \in \cC (f)}{} 1,
  $$
  can be computed in two manners: on one hand, we have two cells per face,
  so this sum is equal to $2 \# \cF$. On the other hand, when doing this
  sum, each cell is visited $4$ times (because each cell has four faces), and
  so the sum is equal to $4 N$.
  This means that
  $$\#\cF = 2N.$$
  We are now interested in the following sum
  $$
  \Sum{p \in \cP}{}
  \Sum{c \in \cC (p)}{} 1,
  $$
  which is both equal to four times the number of points, but also
  four times the number of cells. This means that $\# \cP = N$.
\end{proof}

\subsection{Finite element space notations}

In this article, we will consider continuous and discontinuous finite element
spaces on faces and cells.
We adopt notations close of the ones proposed in
\cite{arnold2014periodic}. 
We will denote by $\PP_k$ the continuous
finite element space on triangles. If the finite element space is
discontinuous, we will denote it by $\dPP_k$. We will also consider
vectorial finite element space, $\bdPP_k$. Last, when needed, we will have
to consider finite element spaces on entities of the mesh
that are not the cells. 
In this case, we will then denote by a parenthesis indicating on which
entity of the mesh the finite element space is defined.
For example, $\dPP_k (\cC)$ is the
discontinuous finite element space of degree $k$ defined on the cells,
whereas $\dPP_k (\cF)$ is the
discontinuous finite element space of degree $k$ defined on the faces. 

The continuous and discontinuous finite element spaces are equipped with the
classical $L^2$ scalar product and its induced norm. In this article, we will
also need to deal with Cartesian products of finite element spaces
of type $\dPP_i ( \cC) \times \dPP_j ( \cF)$, on which we will use the
following scalar product
\begin{equation}
  \label{eq:scalarProduct}
  \scalar{p}{q}_{[\dPP_i ( \cC) \times \dPP_j ( \cF)]} =
  \Sum{c \in \cC }{} \Int{c}{} p_{c} q_c +
  \Sum{f \in \cF}{} \Int{f}{} p_f q_f.
\end{equation}
We define the same type of notations by replacing $\PP$ by $\QQ$ for the case
of Cartesian meshes. 

On Cartesian meshes, we will use enriched versions of
$\bdQQ_k$. For this, we define
$$\QQ_{i,j} = \left\{ p \in \R[x,y] \qquad
\degx p \leq i \quad \text{and} \quad \degy p \leq j
\right\},$$
\answerI{where $\degx$ (resp. $\degy$) is the degree in $x$ (resp. $y$),}
and can define the following cellwise continuous
vectorial finite element space on quads
\begin{equation}
  \label{eq:HDivQuad}
  \bdQQkDiv (\cC)
  = \answerI{ \left(
    \begin{array}{c}
    \dd \QQ_{k,k} + \dd \QQ_{k+1,k-1}  \\
    \dd \QQ_{k,k} + \dd \QQ_{k-1,k+1}
    \end{array}
    \right)
  }\oplus \Vect \left(
  \begin{array}{c}
    - x^{k+1} y^k \\
    x^k y^{k+1}
  \end{array}
  \right)
\end{equation}
that will be suited for the curl/div diagram \eqref{eq:DiagramCurl}, and the following
vectorial finite element space
\begin{equation}
  \label{eq:HCurlQuad}
  \bdQQkCurl (\cC)
  =
  \answerI{
    \left(
    \begin{array}{c}
    \dd \QQ_{k,k} + \dd \QQ_{k-1,k+1} \\
    \dd \QQ_{k,k} + \dd \QQ_{k+1,k-1}
    \end{array}
    \right)}
  \oplus \Vect \left(
  \begin{array}{c}
    x^k y^{k+1} \\
    x^{k+1} y^k \\
  \end{array}
  \right)
\end{equation}
that will be suited for the grad/curl diagram \eqref{eq:DiagramGrad}, and
which is nothing but a $\pi/2$ rotation of~$\bdQQkDiv (\cC)$. 
Note that the space \eqref{eq:HDivQuad} is the discontinuous version of
the space $\boldsymbol{\mathcal{S}}_r$ defined
in \cite[p. 2432]{arnold2005quadrilateral} for ensuring optimal approximation of
vectors on general quadrangular meshes.
It is clear that the cellwise
divergence of $\bdQQkDiv (\cC)$
or the cellwise curl of $\bdQQkCurl (\cC)$ map to
the following finite element space
$$\dQQkHat ( \cC) :=
\dd \QQ_{k-1} ( \cC) + \dd \QQ_{k,k-1} ( \cC) + \dd \QQ_{k-1,k} ( \cC).$$
Last, we will denote by $\KK$ the space of constant elements of the
discretisation of the space $A$, $\bKK$ the space of constant vectors of the
discretisation of the space $\bB$ and $\mathbbm{k}$ the space of
constant elements of $\dPP_k (\mathcal{F})$. 

\section{Finite element spaces inspired by the conformal case}
\label{sec:dRT}

In the conformal case, it is known that the \autoref{def:harmonicGap}
is ensured
for the following complexes
\begin{equation}
  \label{eq:DiagramConformal}
  \left\{
  \begin{array}{l}
    \PP_{k+1}
    \quad
    \diag{\nablaperp}
    \quad
    \RT_{k+1} ^{\triangle} ( \cC )
    \quad
    \diag{\divx}
    \quad
    \dPP_{k} ( \cC) \\
    \QQ_{k+1}
    \quad
    \diag{\nablaperp}
    \quad
    \RT_{k+1} ^{\square} ( \cC )
    \quad
    \diag{\divx}
    \quad
    \dQQ_{k} ( \cC) \\
  \end{array}
  \right. , 
\end{equation}
on both the triangular and quadrangular case. The trace of the Raviart-Thomas
finite element spaces
$\RT_{k+1} ^{\square}$ and $\RT_{k+1} ^{\triangle}$
are known to be of degree $k$ in both the triangular
and quadrangular case. Therefore, by relaxing the
normal continuity constraint, the following complexes may be considered
\begin{equation}
  \label{eq:DiagramNonConformal}
  \left\{
  \begin{array}{l}
    \PP_{k+1}
    \quad
    \diag{\nablaperp}
    \quad
    \dRT_{k+1} ^{\triangle}( \cC )
    \quad
    \diag{\ddivx}
    \quad
    \dPP_{k} ( \cC) \times \dPP_{k} ( \cF) \\
    \QQ_{k+1}
    \quad
    \diag{\nablaperp}
    \quad
    \dRT_{k+1} ^{\square} ( \cC )
    \quad
    \diag{\ddivx}
    \quad
    \dQQ_{k} ( \cC) \times \dPP_{k} ( \cF).
  \end{array}
  \right.
\end{equation}
The discrete maps are defined as follows:
\begin{itemize}
\item $\nablaperp$ is the classical $\nablaperp$ operator: 
  $$\forall c \in \cC \qquad \forall p \in \PP_{k+1} \quad 
  \nablaperp (p)_{|c} = \nablaperp (p_{|c}).$$
\item $\ddivx$ is the divergence where the derivation is taken in the sense of
  distributions:
  $$
  \forall \bu \in \bdPP_{k} \quad
  \left\{
  \begin{array}{l}
    \forall c \in \cC \qquad  
    \ddivx (\bu)_{|c} = \divx (\bu_{|c}) \\
    \forall f \in \cF \qquad  
    \ddivx (\bu)_{|f} = \Jump{\bu \cdot \bn_f}. \\
  \end{array}
  \right.
  $$
\end{itemize}
We first compute the dimension of each of the finite element spaces
of \eqref{eq:DiagramNonConformal}
\begin{prop}[Dimension of the finite element spaces]
  \label{prop:nonconformalDimension}
  If the mesh is triangular and periodic, then
  $$
  \left\{
  \begin{array}{r@{\, = \, }l}
    \dim \PP_{k+1} & \dfrac{N (k+1) ^2}{2}\\
    \dim \dRT_{k+1} ^{\triangle} ( \cC ) & N (k+1)(k+3) \\
    \dim \left( \dPP_k ( \cF ) \times \dPP_{k} ( \cC) \right) &
    \dfrac{N(k+1)(k+5)}{2}.
    \\
  \end{array}
  \right. , 
  $$
  whereas for a Cartesian periodic mesh,
  $$
  \left\{
  \begin{array}{r@{\, = \, }l}
    \dim \QQ_{k+1} & N (k+1) ^2\\
    \dim \dRT_{k+1} ^{\square} & 2 N (k+2)(k+1)\\
    \dim \left( \dPP_k ( \cF ) \times \dQQ_{k} ( \cC) \right) &
    N(k+1)(k+3).
    \\
  \end{array}
  \right.
  $$
\end{prop}

\begin{proof}
  We first address the triangular case.
    A $\PP_{k+1}$ continuous finite element space has
  \begin{itemize}
  \item $1$ degree of freedom on each point. 
  \item $k$ degrees of freedom on each face.
  \item $\dfrac{k(k-1)}{2}$ degrees of freedom inside each cell.
  \end{itemize}
  Adding all these degrees of freedom leads to
  $$
  \begin{array}{r@{\, = \, }l}
    \dim \PP_{k+1} & 1 \times \#\cP + k \#\cF + \dfrac{k(k-1)}{2} \,
    \#\cC \\
    & \dfrac{N}{2} + k \, \dfrac{3N}{2} +
    \dfrac{k(k-1)}{2} \, N
    \\
    & \dfrac{N}{2} \, \left( 1 + 3 k + k(k-1)\right) \\
    & \dfrac{N}{2} \, \left( k^2 + 2 k +1 \right) \\
    \dim \PP_{k+1}
    & \dfrac{N (k+1)^2}{2}. \\
  \end{array}
  $$
  Then the $(k+1)$th order Raviart-Thomas simplicial finite element is known
  for having $(k+1)(k+3)$ degrees of freedom (see e.g.
  \cite[Lemma 14.6 p.137]{ern2021finite}\footnote{Note that regarding
    the notations, what is denoted by $\RT_{k}  ^{\triangle}$
    in \cite{ern2021finite}
    is denoted here $\RT_{k+1}  ^{\triangle}$.}).
  As the space is discontinuous, this
  gives
  $$\dim \dRT_{k+1} ^{\triangle}= N (k+1)(k+3).$$
  It remains to compute the dimension of
  $\dPP_k ( \cF ) \times \dPP_{k} ( \cC)$
  $$
  \begin{array}{r@{\, = \, }l}
    \dim \left( \dPP_k ( \cF ) \times \dPP_{k} ( \cC) \right) &
    (k+1) \# \cF + \dfrac{(k+1)(k+2)}{2} \, \# \cC \\
    & (k+1) \, \dfrac{3 N}{2} + \dfrac{(k+1)(k+2)}{2} \, N \\
    & \dfrac{N}{2} \, (k+1)(k+5).
  \end{array}
  $$
  We are now interested in the dimension of the finite element spaces for
  the quadrangular mesh.
  We begin by computing the dimension of $\QQ_{k+1}$.
  An element of $\QQ_{k+1}$ has
  \begin{itemize}
  \item $1$ degree of freedom at each point,
  \item $k$ degrees of freedom at each face,
  \item $k^2$ degrees of freedom inside each cell.
  \end{itemize}
  Summing all these degrees of freedom and using \autoref{prop:NElementsQuad}
  gives
  $$\dim \QQ_{k+1} = \# \cP + k \# \cF + k^2 \# \cC =
  N + k (2N) + k^2 N = N(k+1)^2.
  $$
  Then the $(k+1)$th order Raviart-Thomas quadrangular finite element is known
  for having $2(k+1)(k+2)$ degrees of freedom (see e.g.
  \cite[Section 14.5.2 p.142]{ern2021finite}
  As the space is discontinuous, this gives
  $$\dim \dRT_{k+1} ^{\square} = N (k+1)(k+3).$$
  It remains to compute the dimension of
  $\dPP_k ( \cF ) \times \dQQ_{k} ( \cC)$
  $$
  \begin{array}{r@{\, = \, }l}
    \dim \left( \dPP_k ( \cF ) \times \dQQ_{k} ( \cC) \right) &
    (k+1) \# \cF + (k+1)^2 \,  \# \cC \\
    & (k+1) \, 2N + (k+1)^2 \, N \\
    & N(k+1)(k+3).
  \end{array}
  $$

\end{proof}

Properties of the complex \eqref{eq:DiagramNonConformal} was addressed in a
more general framework in \cite{licht2017complexes}, and lead in dimension 2
to the following proposition
\begin{prop}
\label{prop:SummaryNonConformalCurlDiv}
The discrete diagram \eqref{eq:DiagramNonConformal}
ensures the \autoref{def:harmonicGap}.
Moreover,
for triangles:
$$
\left\{
\begin{array}{r@{\, = \, }l}
  \PP_{k+1} / \KK & \ker \left( \nablaperp \right)  \\
  \left( \dPP_k ( \cF ) \times \dPP_k ( \cC) \right) / \mathbbm{k} &
  \Range \left( \ddivx \right), 
\end{array}
\right. 
$$
and for quadrangles
$$
\left\{
\begin{array}{r@{\, = \, }l}
  \QQ_{k+1} / \KK & \ker \left( \nablaperp \right)  \\
  \left( \dPP_k ( \cF ) \times \dQQ_k ( \cC) \right) / \mathbbm{k} &
  \Range \left( \ddivx \right).
\end{array}
\right. 
$$
\end{prop}
The location of the degrees of freedom for this discrete de-Rham
complex for Cartesian meshes is summarised in \autoref{fig:deRhamQuadDiv}, and
in \autoref{fig:deRhamTriDivRT} for triangles.

By changing the representation of the linear forms, which is equivalent
to rotating of $\pi/2$ the vector spaces, the following
proposition is also obtained:
\begin{prop}
\label{prop:SummaryNonConformalGradCurl}
The discrete diagram
$$
\left\{
\begin{array}{l}
  \PP_{k+1}
  \quad
  \diag{\gradx}
  \quad
  \dN_{k+1} ^{\triangle}( \cC )
  \quad
  \diag{\dcurlS}
  \quad
  \dPP_{k} ( \cC) \times \dPP_{k} ( \cF) \\
  \QQ_{k+1}
  \quad
  \diag{\gradx}
  \quad
  \dN_{k+1} ^{\square} ( \cC )
  \quad
  \diag{\dcurlS}
  \quad
  \dQQ_{k} ( \cC) \times \dPP_{k} ( \cF).
\end{array}
\right.
$$
where 
$\dcurlS$ is $\curlS$ in the sense of distributions, 
ensures the \autoref{def:harmonicGap}. Moreover,
for triangles:
$$
\left\{
\begin{array}{r@{\, = \, }l}
  \PP_{k+1} / \KK & \ker \left( \gradx \right)  \\
  \left( \dPP_k ( \cF ) \times \dPP_k ( \cC) \right) / \mathbbm{k} &
  \Range \left( \dcurlS \right), 
\end{array}
\right. 
$$
and for quadrangles
$$
\left\{
\begin{array}{r@{\, = \, }l}
  \QQ_{k+1} / \KK & \ker \left( \gradx \right)  \\
  \left( \dPP_k ( \cF ) \times \dQQ_k ( \cC) \right) / \mathbbm{k} &
  \Range \left( \dcurlS \right).
\end{array}
\right. 
$$
\end{prop}
The location of the degrees of freedom for this discrete de-Rham
complex for Cartesian meshes is summarised in \autoref{fig:deRhamQuadCurl} and
in \autoref{fig:deRhamTriCurlN} for triangles.

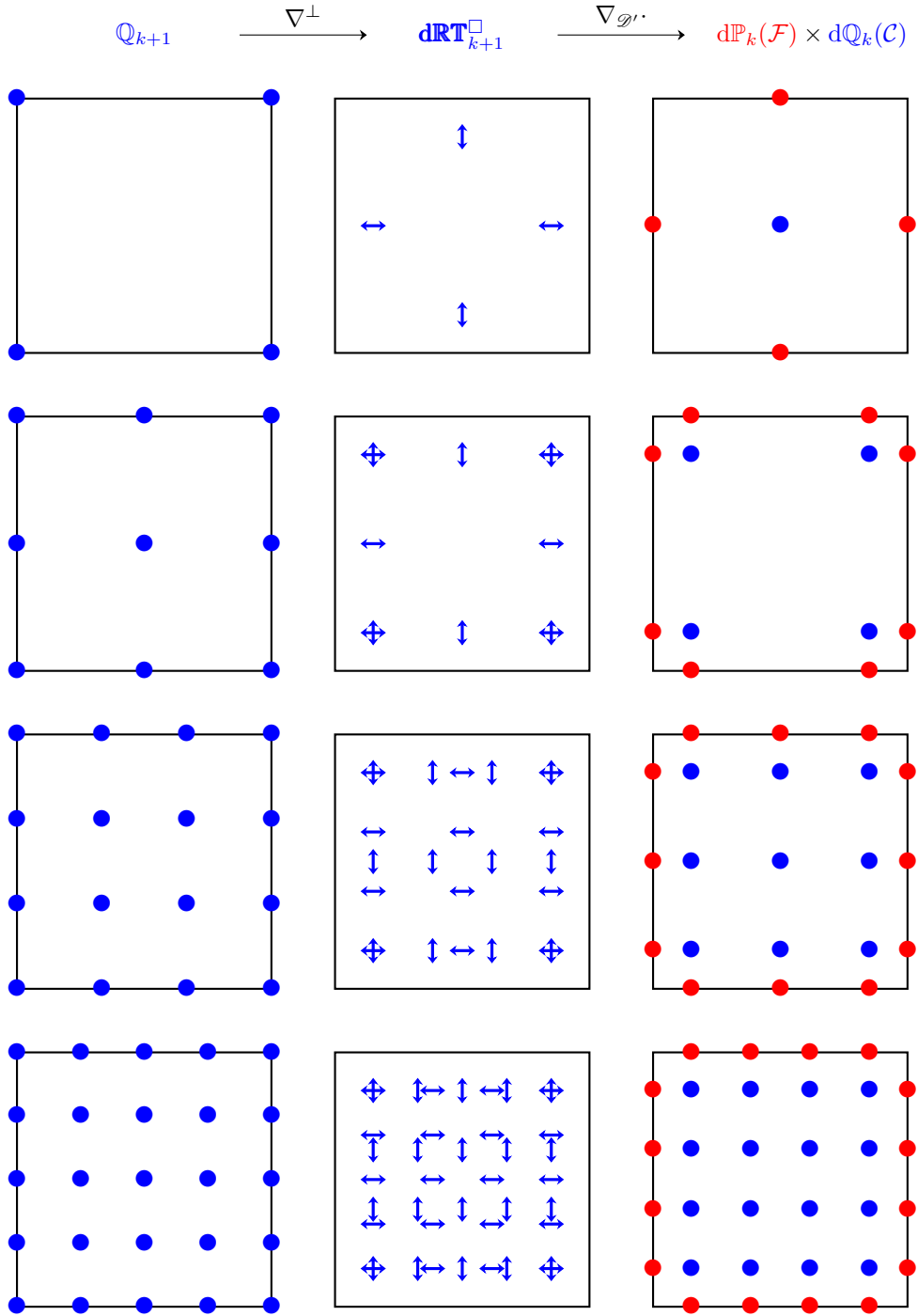
\begin{figure}
  \begin{center}
    \begin{tikzpicture}[scale=0.9]

      \node at (2,15) {\color{blue} $\QQ_{k+1}$};
      \draw[->] (3.5,15) -- (5.5,15) node [midway, above] {$\nablaperp$};
      \node at (7,15) {\color{blue}$\dRT_{k+1} ^{\square}$};
      \draw[->] (8.5,15) -- (10.5,15) node [midway, above] {$\ddivx$};
      \node at (12.5,15) {$  {\color{red} \dPP_k ( \cF )} \times {\color{blue} \dQQ_{k} ( \cC)}$};

      \LagrangeQuad{2}{12}{4}{1};
      \LagrangeQuad{2}{7}{4}{2};
      \LagrangeQuad{2}{2}{4}{3};
      \LagrangeQuad{2}{-3}{4}{4};

      \node[rectangle,draw, minimum size=3.6cm,thick] (r) at (7,12) {};
      \draw[color=blue,<->,very thick] (7,-0.2+12-0.5*2.8) -- (7,0.2+12-0.5*2.8);
      \draw[color=blue,<->,very thick] (7,-0.2+12+0.5*2.8) -- (7,0.2+12+0.5*2.8);
      \draw[color=blue,<->,very thick] (-0.2+7-0.5*2.8,12) -- (0.2+7-0.5*2.8,12);
      \draw[color=blue,<->,very thick] (-0.2+7+0.5*2.8,12) -- (0.2+7+0.5*2.8,12);
      \dRaviartThomasQuad{7}{7}{2.8}{1};
      \dRaviartThomasQuad{7}{2}{2.8}{2};
      \dRaviartThomasQuad{7}{-3}{2.8}{3};

      \node[rectangle,draw, minimum size=3.6cm,thick] (r) at (12,12) {};
      \node[color=blue] at (12,12) {\LARGE $\bullet$};
      \node[color=red] at (12-0.5*4,12) {\LARGE $\bullet$};
      \node[color=red] at (12+0.5*4,12) {\LARGE $\bullet$};
      \node[color=red] at (12,12-0.5*4) {\LARGE $\bullet$};
      \node[color=red] at (12,12+0.5*4) {\LARGE $\bullet$};
      
      \dFaceCellsQuad{12}{7}{2.8}{1};
      \dFaceCellsQuad{12}{2}{2.8}{2};
      \dFaceCellsQuad{12}{-3}{2.8}{3};


    \end{tikzpicture}
  \end{center}
  \caption{\label{fig:deRhamQuadDiv} Representation of the degrees of
    freedom of the finite element
    spaces involved in the curl/div de-Rham complex for Cartesian meshes
    for $k=0,1,2$ and $3$.
    Points denote scalar degrees of freedom, whereas arrows denote vectorial degrees of freedom. Both scalar
  and vectorial volume degrees of freedom are represented in blue, whereas the degrees of freedom in the face finite element space are represented in red.}
\end{figure}

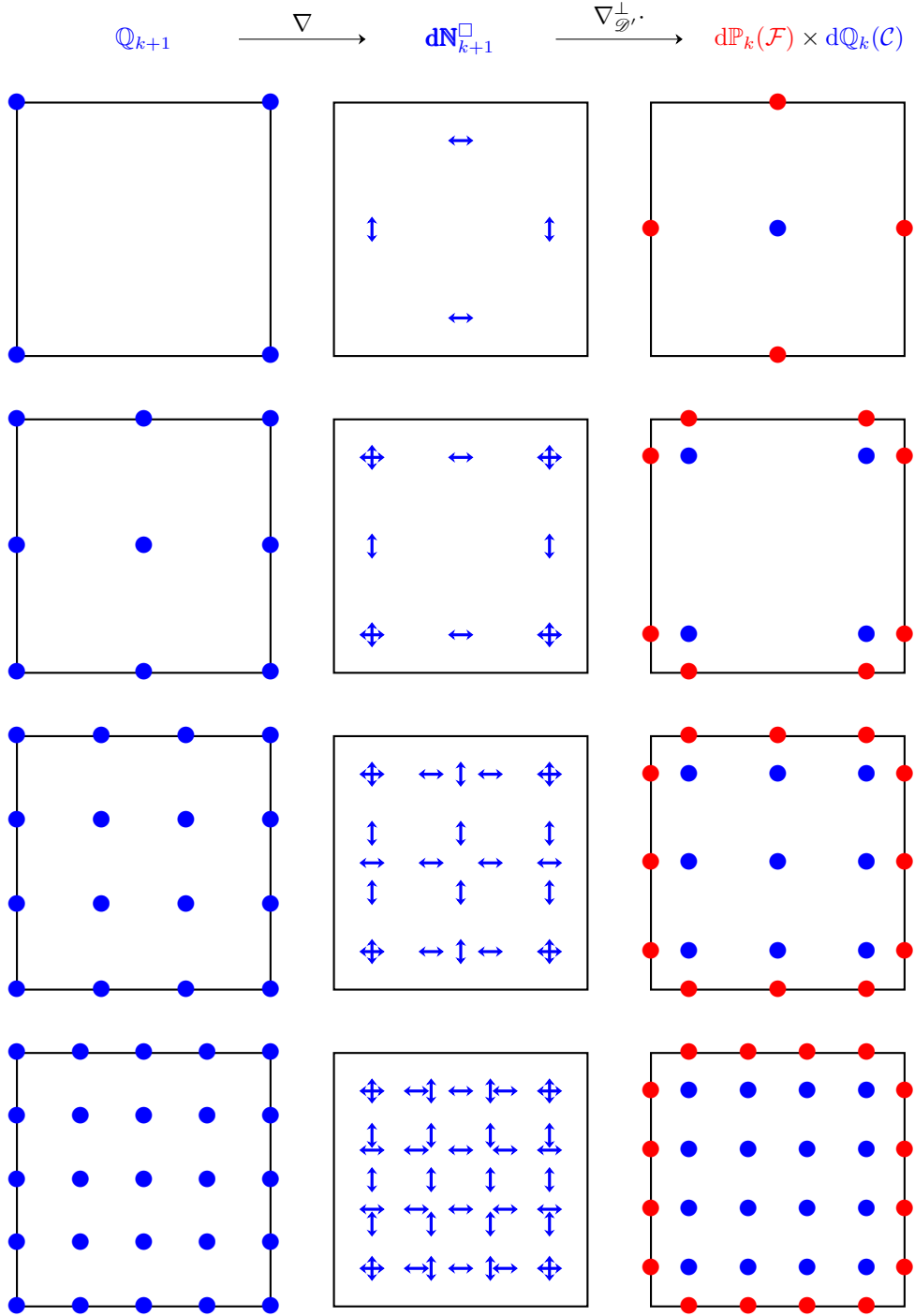
\begin{figure}
  \begin{center}
    \begin{tikzpicture}[scale=0.9]

      \node at (2,15) {\color{blue} $\QQ_{k+1}$};
      \draw[->] (3.5,15) -- (5.5,15) node [midway, above] {$\gradx$};
      \node at (7,15) {\color{blue}$\dN_{k+1} ^{\square}$};
      \draw[->] (8.5,15) -- (10.5,15) node [midway, above] {$\dcurlS$};
      \node at (12.5,15) {$  {\color{red} \dPP_k ( \cF )} \times {\color{blue} \dQQ_{k} ( \cC)}$};

      \LagrangeQuad{2}{12}{4}{1};
      \LagrangeQuad{2}{7}{4}{2};
      \LagrangeQuad{2}{2}{4}{3};
      \LagrangeQuad{2}{-3}{4}{4};

      \node[rectangle,draw, minimum size=3.6cm,thick] (r) at (7,12) {};
      \draw[color=blue,<->,very thick] (-0.2+7,12-0.5*2.8) -- (0.2+7,12-0.5*2.8);
      \draw[color=blue,<->,very thick] (-0.2+7,12+0.5*2.8) -- (0.2+7,12+0.5*2.8);
      \draw[color=blue,<->,very thick] (7-0.5*2.8,-0.2+12) -- (7-0.5*2.8,0.2+12);
      \draw[color=blue,<->,very thick] (7+0.5*2.8,-0.2+12) -- (7+0.5*2.8,0.2+12);
      \dNedelecQuad{7}{7}{2.8}{1};
      \dNedelecQuad{7}{2}{2.8}{2};
      \dNedelecQuad{7}{-3}{2.8}{3};

      \node[rectangle,draw, minimum size=3.6cm,thick] (r) at (12,12) {};
      \node[color=blue] at (12,12) {\LARGE $\bullet$};
      \node[color=red] at (12-0.5*4,12) {\LARGE $\bullet$};
      \node[color=red] at (12+0.5*4,12) {\LARGE $\bullet$};
      \node[color=red] at (12,12-0.5*4) {\LARGE $\bullet$};
      \node[color=red] at (12,12+0.5*4) {\LARGE $\bullet$};
      
      \dFaceCellsQuad{12}{7}{2.8}{1};
      \dFaceCellsQuad{12}{2}{2.8}{2};
      \dFaceCellsQuad{12}{-3}{2.8}{3};
      
    \end{tikzpicture}
  \end{center}
  \caption{\label{fig:deRhamQuadCurl} Representation of the degrees of
    freedom of the finite element
    spaces involved in the grad/curl de-Rham complex for Cartesian meshes
    for $k=0,1,2$ and $3$.
    Same code for colors as in \autoref{fig:deRhamQuadDiv} is used.}
\end{figure}

\begin{figure}
  \begin{center}
    \begin{tikzpicture}[scale=0.9]
      \def\LagrangeTriangle#1#2#3#4{
        \draw[-,thick] (#1-#3/3,#2-#3/3) -- (#1+2*#3/3,#2-#3/3) -- (#1-#3/3,#2+2*#3/3) -- cycle;
        \foreach \x in {0,...,#4}{
          \pgfmathsetmacro\maxy{#4-\x};
  
          \foreach \y in {0,...,\maxy}{
            \pgfmathsetmacro\xx{(#1-#3/3)*(1-\x/#4-\y/#4)+(#1+2*#3/3)*\y/#4+(#1-#3/3)*\x/#4};
            \pgfmathsetmacro\yy{(#2-#3/3)*(1-\x/#4-\y/#4)+(#2-#3/3)*\y/#4+(#2+2*#3/3)*\x/#4};
            \node[color=blue] at (\xx,\yy) {\LARGE $\bullet$};
          }
        }
      }

      \node at (2,15.5) {\color{blue} $\PP_{k+1}$};
      \draw[->] (3.5,15.5) -- (5.5,15.5) node [midway, above] {$\nablaperp$};
      \node at (7,15.5) {\color{blue}$\dRT_{k+1} ^{\triangle}$};
      \draw[->] (8.5,15.5) -- (10.5,15.5) node [midway, above] {$\ddivx$};
      \node at (12.5,15.5) {$  {\color{red} \dPP_k ( \cF )} \times {\color{blue} \dPP_{k} ( \cC)}$};
      \LagrangeTriangle{2}{12}{4}{1};
      \LagrangeTriangle{2}{7}{4}{2};
      \LagrangeTriangle{2}{2}{4}{3};
      \LagrangeTriangle{2}{-3}{4}{4};

      \def\RTVectorDG#1#2#3#4{
        \draw[-,thick] (#1-4/3,#2-4/3) -- (#1+2*4/3,#2-4/3) -- (#1-4/3,#2+2*4/3) -- cycle;
        \pgfmathsetmacro\maxm{#4-1};
        \xifnum{ \maxm > 0 }{
          \foreach \x in {0,...,\maxm}{
            \pgfmathsetmacro\maxy{\maxm-\x};
            \foreach \y in {0,...,\maxy}{
              \pgfmathsetmacro\xx{(#1-#3/3)*(1-\x/\maxm-\y/\maxm)+(#1+2*#3/3)*\y/\maxm+(#1-#3/3)*\x/\maxm};
              \pgfmathsetmacro\yy{(#2-#3/3)*(1-\x/\maxm-\y/\maxm)+(#2-#3/3)*\y/\maxm+(#2+2*#3/3)*\x/\maxm};
              \draw[<->,color=blue,very thick] (\xx-0.2,\yy) -- (\xx+0.2,\yy);
              \draw[<->,color=blue,very thick] (\xx,\yy-0.2) -- (\xx,\yy+0.2);
            }
          }
        }
        {}
        \pgfmathsetmacro\facta{0.87};
        \pgfmathsetmacro\maxxa{#4+1};
        \foreach \xa in {1,...,\maxxa}{
          \pgfmathsetmacro\xxx{#1-4/3*\facta + 4*\xa/(\maxxa+1)*\facta};
          \pgfmathsetmacro\yya{#2-4/3*\facta};
          \draw[<->,color=blue,very thick] (\xxx,\yya-0.2) -- (\xxx,\yya+0.2);
          \pgfmathsetmacro\yyy{#2-4/3*\facta + 4*\xa/(\maxxa+1)*\facta};
          \pgfmathsetmacro\xxa{#1-4/3*\facta};
          \draw[<->,color=blue,very thick] (\xxa-0.2,\yyy) -- (\xxa+0.2,\yyy);
          \pgfmathsetmacro\xxxx{#1 + (2*4/3*\facta) - 4*\xa/(\maxxa+1)*\facta};
          \pgfmathsetmacro\yyyy{#2 -   (4/3*\facta) + 4*\xa/(\maxxa+1)*\facta};
          \pgfmathsetmacro\deltal{sqrt(2)/2*0.2};
          \draw[<->,color=blue,very thick] (\xxxx-\deltal,\yyyy-\deltal) -- (\xxxx+\deltal,\yyyy+\deltal);
        }
      }

      \RTVectorDG{7}{12}{2.8}{0};
      \RTVectorDG{7}{7}{2.8}{1};
      \draw[<->,color=blue,very thick] (7-0.2,7) -- (7+0.2,7);
      \draw[<->,color=blue,very thick] (7,7-0.2) -- (7,7+0.2);
      \RTVectorDG{7}{2}{2.8}{2};
      \RTVectorDG{7}{-3}{2.8}{3};

      \def\dFaceCellsTriangle#1#2#3#4{
        \draw[-,thick] (#1-4/3,#2-4/3) -- (#1+2*4/3,#2-4/3) -- (#1-4/3,#2+2*4/3) -- cycle;
        \pgfmathsetmacro\newvar{#4+1};
        \xifnum{ \newvar < 0}
                   {
                   }
                   {
                     \xifnum{ \newvar > 0}{
                       \foreach \x in {0,...,\newvar}{
                         \pgfmathsetmacro\maxy{\newvar-\x};
  
                         \foreach \y in {0,...,\maxy}{
                           \pgfmathsetmacro\xx{(#1-#3/3)*(1-\x/\newvar-\y/\newvar)+(#1+2*#3/3)*\y/\newvar+(#1-#3/3)*\x/\newvar};
                           \pgfmathsetmacro\yy{(#2-#3/3)*(1-\x/\newvar-\y/\newvar)+(#2-#3/3)*\y/\newvar+(#2+2*#3/3)*\x/\newvar};
                           \node[color=blue] at (\xx,\yy) {\LARGE $\bullet$};
                         }
                       }
                     }
                     {
                       \node[color=blue] at (#1,#2) {\LARGE $\bullet$};
                     }
                   }

                   \pgfmathsetmacro\maxxa{#4+2};
                   \foreach \xa in {1,...,\maxxa}{
                     \pgfmathsetmacro\xxx{#1-4/3 + 4*\xa/(\maxxa+1)};
                     \pgfmathsetmacro\yyy{#2-4/3 + 4*\xa/(\maxxa+1)};
                     \node[color=red] at (\xxx,#2-4/3) {\LARGE $\bullet$};
                     \node[color=red] at (#1-4/3,\yyy) {\LARGE $\bullet$};
                     \pgfmathsetmacro\xxxx{#1 + (2*4/3) - 4*\xa/(\maxxa+1)};
                     \pgfmathsetmacro\yyyy{#2 -   (4/3) + 4*\xa/(\maxxa+1)};
                     \node[color=red] at (\xxxx,\yyyy) {\LARGE $\bullet$};
                   }
                   
      }
      \dFaceCellsTriangle{12}{12}{2.8}{-1};
      \dFaceCellsTriangle{12}{7}{2.8}{0};
      \dFaceCellsTriangle{12}{2}{2.8}{1};
      \dFaceCellsTriangle{12}{-3}{2.8}{2};

    \end{tikzpicture}
  \end{center}
  \caption{\label{fig:deRhamTriDivRT}
    Representation of the finite element spaces involved in the
    curl/div de-Rham complex for triangular meshes for $k=0,1,2$ and $3$.
    Same code for colors as in \autoref{fig:deRhamQuadDiv} is used.
  }
\end{figure}
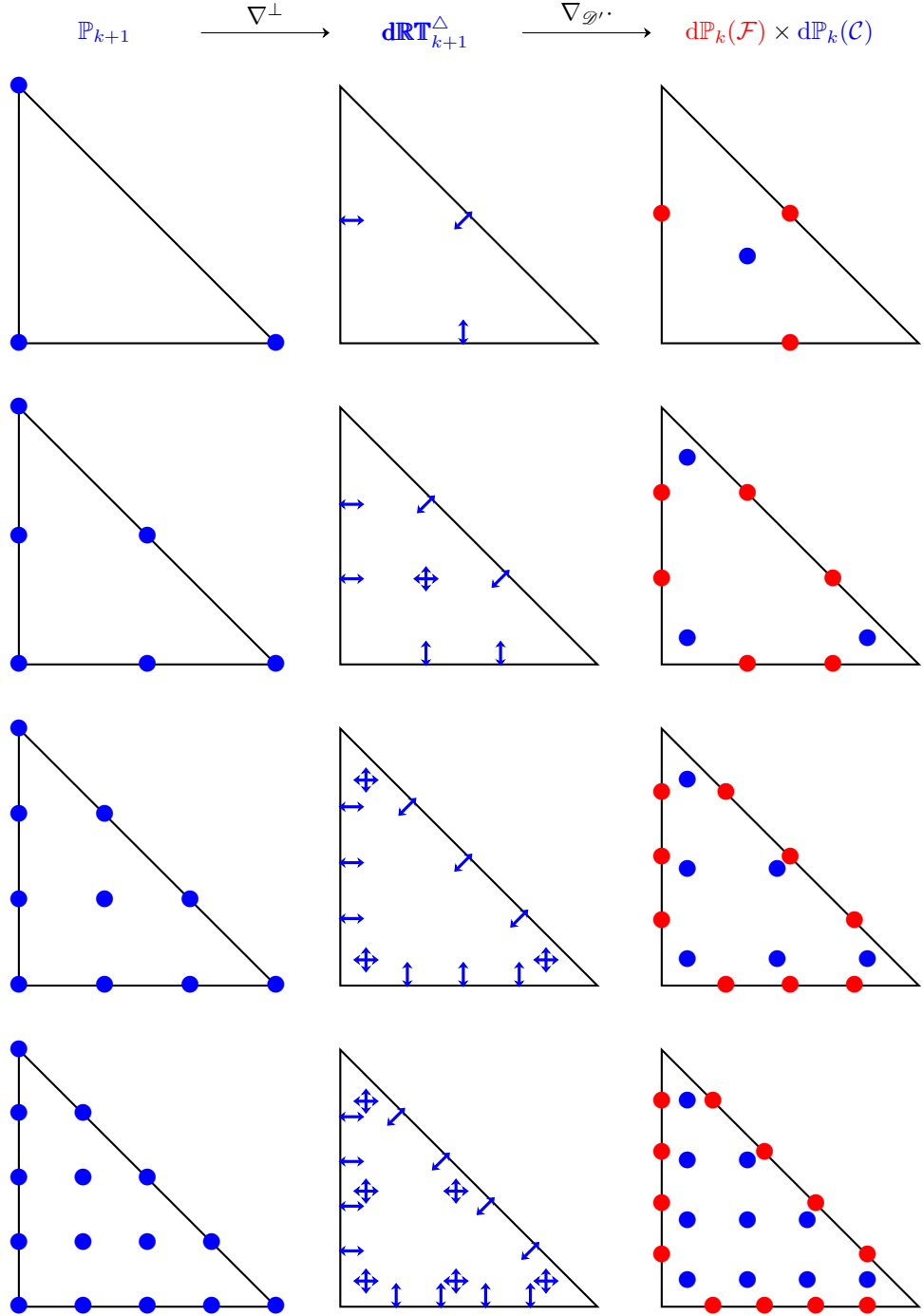

\begin{figure}
  \begin{center}
    \begin{tikzpicture}[scale=0.9]
      \def\LagrangeTriangle#1#2#3#4{
        \draw[-,thick] (#1-#3/3,#2-#3/3) -- (#1+2*#3/3,#2-#3/3) -- (#1-#3/3,#2+2*#3/3) -- cycle;
        \foreach \x in {0,...,#4}{
          \pgfmathsetmacro\maxy{#4-\x};
  
          \foreach \y in {0,...,\maxy}{
            \pgfmathsetmacro\xx{(#1-#3/3)*(1-\x/#4-\y/#4)+(#1+2*#3/3)*\y/#4+(#1-#3/3)*\x/#4};
            \pgfmathsetmacro\yy{(#2-#3/3)*(1-\x/#4-\y/#4)+(#2-#3/3)*\y/#4+(#2+2*#3/3)*\x/#4};
            \node[color=blue] at (\xx,\yy) {\LARGE $\bullet$};
          }
        }
      }

      \node at (2,15.5) {\color{blue} $\PP_{k+1}$};
      \draw[->] (3.5,15.5) -- (5.5,15.5) node [midway, above] {$\gradx$};
      \node at (7,15.5) {\color{blue}$\dN_{k+1} ^{\triangle}$};
      \draw[->] (8.5,15.5) -- (10.5,15.5) node [midway, above] {$\dcurlS$};
      \node at (12.5,15.5) {$  {\color{red} \dPP_k ( \cF )} \times {\color{blue} \dPP_{k} ( \cC)}$};
      \LagrangeTriangle{2}{12}{4}{1};
      \LagrangeTriangle{2}{7}{4}{2};
      \LagrangeTriangle{2}{2}{4}{3};
      \LagrangeTriangle{2}{-3}{4}{4};

      \def\NVectorDG#1#2#3#4{
        \draw[-,thick] (#1-4/3,#2-4/3) -- (#1+2*4/3,#2-4/3) -- (#1-4/3,#2+2*4/3) -- cycle;
        \pgfmathsetmacro\maxm{#4-1};
        \xifnum{ \maxm > 0 }{
          \foreach \x in {0,...,\maxm}{
            \pgfmathsetmacro\maxy{\maxm-\x};
            \foreach \y in {0,...,\maxy}{
              \pgfmathsetmacro\xx{(#1-#3/3)*(1-\x/\maxm-\y/\maxm)+(#1+2*#3/3)*\y/\maxm+(#1-#3/3)*\x/\maxm};
              \pgfmathsetmacro\yy{(#2-#3/3)*(1-\x/\maxm-\y/\maxm)+(#2-#3/3)*\y/\maxm+(#2+2*#3/3)*\x/\maxm};
              \draw[<->,color=blue,very thick] (\xx-0.2,\yy) -- (\xx+0.2,\yy);
              \draw[<->,color=blue,very thick] (\xx,\yy-0.2) -- (\xx,\yy+0.2);
            }
          }
        }
        {}
        \pgfmathsetmacro\facta{0.87};
        \pgfmathsetmacro\maxxa{#4+1};
        \foreach \xa in {1,...,\maxxa}{
          \pgfmathsetmacro\xxx{#1-4/3*\facta + 4*\xa/(\maxxa+1)*\facta};
          \pgfmathsetmacro\yya{#2-4/3*\facta};
          \draw[<->,color=blue,very thick] (\xxx-0.2,\yya) -- (\xxx+0.2,\yya);
          \pgfmathsetmacro\yyy{#2-4/3*\facta + 4*\xa/(\maxxa+1)*\facta};
          \pgfmathsetmacro\xxa{#1-4/3*\facta};
          \draw[<->,color=blue,very thick] (\xxa,\yyy-0.2) -- (\xxa,\yyy+0.2);
          \pgfmathsetmacro\xxxx{#1 + (2*4/3*\facta) - 4*\xa/(\maxxa+1)*\facta};
          \pgfmathsetmacro\yyyy{#2 -   (4/3*\facta) + 4*\xa/(\maxxa+1)*\facta};
          \pgfmathsetmacro\deltal{sqrt(2)/2*0.2};
          \draw[<->,color=blue,very thick] (\xxxx+\deltal,\yyyy-\deltal) -- (\xxxx-\deltal,\yyyy+\deltal);
        }
      }

      \NVectorDG{7}{12}{2.8}{0};
      \NVectorDG{7}{7}{2.8}{1};
      \draw[<->,color=blue,very thick] (7-0.2,7) -- (7+0.2,7);
      \draw[<->,color=blue,very thick] (7,7-0.2) -- (7,7+0.2);
      \NVectorDG{7}{2}{2.8}{2};
      \NVectorDG{7}{-3}{2.8}{3};

      \def\dFaceCellsTriangle#1#2#3#4{
        \draw[-,thick] (#1-4/3,#2-4/3) -- (#1+2*4/3,#2-4/3) -- (#1-4/3,#2+2*4/3) -- cycle;
        \pgfmathsetmacro\newvar{#4+1};
        \xifnum{ \newvar < 0}
                   {
                   }
                   {
                     \xifnum{ \newvar > 0}{
                       \foreach \x in {0,...,\newvar}{
                         \pgfmathsetmacro\maxy{\newvar-\x};
  
                         \foreach \y in {0,...,\maxy}{
                           \pgfmathsetmacro\xx{(#1-#3/3)*(1-\x/\newvar-\y/\newvar)+(#1+2*#3/3)*\y/\newvar+(#1-#3/3)*\x/\newvar};
                           \pgfmathsetmacro\yy{(#2-#3/3)*(1-\x/\newvar-\y/\newvar)+(#2-#3/3)*\y/\newvar+(#2+2*#3/3)*\x/\newvar};
                           \node[color=blue] at (\xx,\yy) {\LARGE $\bullet$};
                         }
                       }
                     }
                     {
                       \node[color=blue] at (#1,#2) {\LARGE $\bullet$};
                     }
                   }

                   \pgfmathsetmacro\maxxa{#4+2};
                   \foreach \xa in {1,...,\maxxa}{
                     \pgfmathsetmacro\xxx{#1-4/3 + 4*\xa/(\maxxa+1)};
                     \pgfmathsetmacro\yyy{#2-4/3 + 4*\xa/(\maxxa+1)};
                     \node[color=red] at (\xxx,#2-4/3) {\LARGE $\bullet$};
                     \node[color=red] at (#1-4/3,\yyy) {\LARGE $\bullet$};
                     \pgfmathsetmacro\xxxx{#1 + (2*4/3) - 4*\xa/(\maxxa+1)};
                     \pgfmathsetmacro\yyyy{#2 -   (4/3) + 4*\xa/(\maxxa+1)};
                     \node[color=red] at (\xxxx,\yyyy) {\LARGE $\bullet$};
                   }
                   
      }
      \dFaceCellsTriangle{12}{12}{2.8}{-1};
      \dFaceCellsTriangle{12}{7}{2.8}{0};
      \dFaceCellsTriangle{12}{2}{2.8}{1};
      \dFaceCellsTriangle{12}{-3}{2.8}{2};

    \end{tikzpicture}
  \end{center}
  \caption{\label{fig:deRhamTriCurlN}
    Representation of the finite element spaces involved in the
    grad/curl de-Rham complex for triangular
    meshes for $k=0,1,2$ and $3$.
    Same code for as in \autoref{fig:deRhamQuadDiv} is used.
  }
\end{figure}
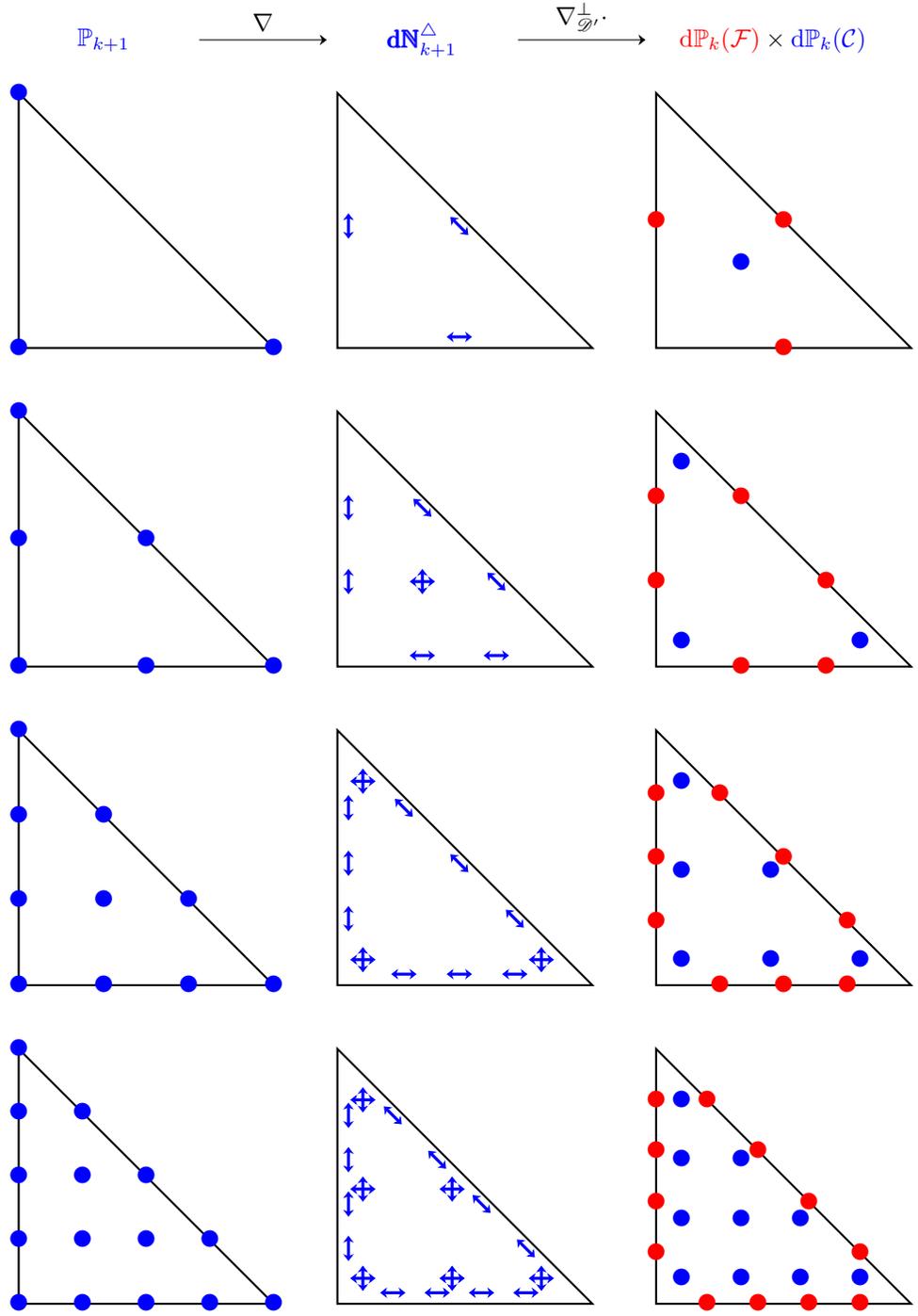

In this section, results of \cite{licht2017complexes} for
discontinuous finite element spaces for vectors have been recalled to
ensure the \autoref{def:harmonicGap}. Still, as the spaces
$\dRT$ and $\dN$ are
obtained by relaxing the normal continuity constraint of the classical
conformal finite element spaces $\RT$ and $\bN$, their number of degrees of
freedom are not optimal. In the following sections, we will try to develop
vector finite element spaces with a lower number of degrees of freedom for
which the \autoref{def:harmonicGap} holds also, by beginning by the triangular
meshes case.

\section{The triangular mesh case}
\label{sec:triangularMesh}

In this section, we are interested in the following diagram
\begin{equation}
  \label{eq:DiagramTri}
  \PP_{k+1}
  \quad
  \diag{\nablaperp}
  \quad
  \bdPP_k ( \cC )
  \quad
  \diag{\ddivx}
  \quad
  \dPP_k ( \cF ) \times \dPP_{k-1} ( \cC).
\end{equation}

\subsection{Dimension of the finite elements spaces}
We first compute the dimension of each of the finite element spaces involved
in \eqref{eq:DiagramTri}, induced by the number
of faces, points and cells that was computed in \autoref{prop:NElementsTri}.

\begin{prop}[Dimension of the finite element spaces]
  \label{prop:dimFETri}
  If the mesh is triangular and periodic, then
  $$
  \left\{
  \begin{array}{r@{\, = \, }l}
    \dim \PP_{k+1} & \dfrac{N (k+1) ^2}{2}\\
    \dim \bdPP_k ( \cC ) & N (k+1)(k+2)\\
    \dim \left( \dPP_k ( \cF ) \times \dPP_{k-1} ( \cC) \right) &
    \dfrac{(k+1)(k+3) N}{2}.
    \\
  \end{array}
  \right.
  $$

\end{prop}

\begin{proof}
  The dimension of $\PP_{k+1}$ was already proven in
  \autoref{prop:nonconformalDimension}.
  We are now interested in the dimension of $\bdPP_k (\cC)$. This
  space is a vector space, and so is composed 
  of two components, each of these components having
  $\dfrac{(k+1)(k+2)}{2}$ degrees of freedom on each cell. This gives
  $$\dim \bdPP_k (\cC) = 2 N \times \dfrac{(k+1)(k+2)}{2}
  =N(k+1)(k+2).$$
  We are finally interested in the dimension of
  $\dPP_k (\cF) \times \dPP_{k-1} ( \cC  )$. This finite element space
  includes $k+1$ degrees of freedom on each face, and
  $\dfrac{k(k+1)}{2}$ degrees of freedom on each cell. Adding all these degrees
  of freedom leads to
  $$
  \begin{array}{r@{\, = \, }l}
    \dim \left( \dPP_k (\cF) \times \dPP_{k-1} ( \cC  ) \right)
    & (k+1) \# \cF + \dfrac{k(k+1)}{2} \, \# \cC\\
    & (k+1) \, \dfrac{3N}{2} + \dfrac{k(k+1)}{2} \, N \\
    \dim \left( \dPP_k (\cF) \times \dPP_{k-1} ( \cC  ) \right) 
    &
    \dfrac{(k+1)(k+3) N}{2} ,  \\
  \end{array}
  $$
  which ends the proof of this proposition. 
\end{proof}

\subsection{Study of $\nablaperp$}

We are now interested in the study of the $\nablaperp$ operator.
\begin{prop}[$\nablaperp$ in the triangular case]
  \label{prop:nablaPerpTri}
  We have
  $$
  \left\{
  \begin{array}{r@{\, = \, }l}
    \dim \left( \ker \nablaperp \right) & 1 \\
    \rank \left( \nablaperp \right)     & \dfrac{N (k+1) ^2}{2} - 1. \\
  \end{array}
  \right.
  $$
\end{prop}
\begin{proof}

  Suppose that a $\psi \in \PP_{k+1}$ is such that $\nablaperp \psi = 0$.
  Then $\partial_x \psi = \partial_y \psi = 0$, so that $\psi$ is
  piecewise constant. But as $\psi$ is continuous, it is actually uniform:
  $$\ker \nablaperp = \KK.$$
  This gives $\dim \ker \nablaperp = 1$, as we are working on a
  domain with a single connected component. Applying the rank-nullity theorem in
  the triangular case
  $$\dim \PP_k = \dim \ker \nablaperp + \rank \left( \nablaperp \right),$$
  leads to
  $$\rank \left( \nablaperp \right) = \dim \PP_k - 1
  = \dfrac{N (k+1) ^2}{2} - 1.$$
  In the same manner, we get, in the quadrangular case:
  $$\rank \left( \nablaperp \right) 
  = N (k+1) ^2 - 1.$$

\end{proof}

\subsection{Discrete divergence free polynomials on the reference cell}

We consider the following application
\begin{equation}
  \bC ^{\partial} _k \, : \, \psi \in \dPP_{k+1} (\hat{K}) \, \longmapsto \,
  \bTr \left( \nablaperp \psi \right) \in \dd \PP_k ( \partial \hat{K} ),
\end{equation}
\answerI{where $\hat{K}$ is the reference triangle.}

\begin{prop}
  \label{prop:CkTri}
  We denote by $\mathbbm{k}$ the constant elements of
  $\dd \PP_k ( \partial \hat{K} )$. Then 
  $$\dd \PP_k ( \partial \hat{K} ) = \Range \bC^{\partial} _k \oplus \mathbbm{k},$$
  where the sum is orthogonal. 
\end{prop}

\begin{proof}
  We denote by $c$ an element of $\mathbbm{k}$. We denote also by $c$ the
  function equal to $c$ on $\hat{K}$. We also denote by
  $u$ an element of $\Range \bC^{\partial}$. Then a $\psi$ exists such
  that $u = \bC^{\partial} _k\left( \psi \right)$. Then
  $$
  \begin{array}{r@{\, = \, }l}
    \Int{\partial \hat{K}}{}  u c  & 
    \Int{\partial \hat{K}}{} \bTr \left( \nablaperp \psi \right) c \\
    &  \Int{\hat{K}}{} \divx \left( c \nablaperp \psi \right)\\
    & \Int{\hat{K}}{} c \, \divx \left( \nablaperp \psi \right)
    +
    \Int{\hat{K}}{} \nabla c \cdot \nablaperp \psi \\
    \Int{\partial \hat{K}}{}  u c 
     & 0,
  \end{array}
  $$
  because $c$ is constant and $\divx \left( \nablaperp \psi \right) = 0$.
  We thus have proven that the sum
  $$\Range \bC^{\partial} _k + \mathbbm{k} , $$
  is direct and orthogonal.

  We are now interested in the study of the kernel of $\bC^{\partial} _k$.
  Suppose that an element $\psi$ is such that $\bC^{\partial} _k (\psi) = 0$.
  We consider the classical Lagrange basis of $\dPP_{k+1} (\hat{K})$.
  Then $\psi$
  is such that its value on the boundary
  of $\hat{K}$ is constant, and may take any value on the degrees of freedom
  matching with the interior nodes. This means that
  $$\dim \left( \ker \bC^{\partial} _k \right) = 1 + \dfrac{k(k-1)}{2}. $$
  Using the rank-nullity theorem gives
  $$
  \begin{array}{r@{\, = \, }l}
  \rank \left( \bC^{\partial} _k \right)
  &
  \dim \dPP_{k+1} - \left( 1 + \dfrac{k(k-1)}{2} \right) \\
  &
  \dfrac{(k+2)(k+3)}{2} -  1 - \dfrac{k(k-1)}{2}  \\
  &
  \dfrac{k^2 + 5k +6 - k^2 + k - 2}{2} \\
  &
  \dfrac{6k + 4}{2} \\
  \rank \left( \bC^{\partial} _k \right)
  & 3 k +2. \\
  \end{array}
  $$
  We also know that $\dim \dd \PP_k \left( \partial \hat{K} \right) = 3(k+1)$. We have
  then $\Range \bC^{\partial} _k \oplus \mathbbm{k}
  \subset \dd \PP_k \left( \partial \hat{K} \right)$, and
  $\dim \left( \Range \bC^{\partial} _k \oplus \mathbbm{k} \right)  =
  \dim \dd \PP_k \left( \partial \hat{K} \right)$, so that
  $\Range \bC^{\partial} _k \oplus \mathbbm{k}  =
  \dd \PP_k \left( \partial \hat{K} \right)$, which ends the proof. 
\end{proof}

\begin{prop}[Decomposition of divergence free elements]
  \label{prop:decompositionTri}
  We denote by $\mathscr{L}^{f,i}$  the Legendre polynomial of
  degree $i$
  on the face $f$ of $\hat{K}$, normalised such that
  $$\Int{\partial \hat{K}}{} \left( \mathscr{L}^{f,i} \right)^2 = 1.$$
  
  Suppose that $\bu \in \bdPP_k (\hat{K})$ is divergence free. Then
  $\bu$ can be uniquely decomposed as
  \begin{equation}
    \label{eq:DecompositionTri}
    \bu = {\bar \bv_{\bu}} ^0 + {\bar \bv_{\bu}} ^1
    + \Sum{f \in \cF (\hat{K})}{} \Sum{i=1}{k}
    {\bar \bv }_{\bu} ^{f,i},
  \end{equation}
  where
  \begin{itemize}
  \item $\answerI{\bar  \bv}_{\bu} ^0$ is in the set
    $$\Phi_k = \left\{
    \bu \in \bdPP_k \qquad \divx \bu = 0 \qquad \bTr(\bu) = 0
    \right\}.$$
  \item $\answerI{\bar  \bv}_{\bu} ^1$ is constant.
  \item ${\bar \bv }_{\bu} ^{f,i}$ is orthogonal to $\Phi_k$ and such that
    \begin{itemize}
    \item $ \divx {\bar \bv }_{\bu} ^{f,i} = 0$. 
    \item $\bTr \left( {\bar \bv }_{\bu} ^{f,i} \right)$ is orthogonal to all
      $\mathscr{L}^{g,j}$ for $\left\{g,j\right\} \neq \left\{f,i\right\}$.
    \end{itemize}
  \end{itemize}
\end{prop}
Note that the sum over the integers $i$ in the sum of ${\bar \bv }_{\bu} ^{f,i}$
in \eqref{eq:DecompositionTri} begins at $1$ and not $0$ for excluding the
ones that would have a constant trace. 
Note also that the set $\Phi_k$ is the same as in the decomposition used
in \cite[Proposition 3.1]{brezzi2012mixed} for the derivation of 
classical conformal divergence free finite elements of Raviart-Thomas
\cite{raviart1977mixed,raviart1977primal}
or Brezzi-Douglas-Marini \cite{brezzi1985two} types (these are respectively
referred as $RT$ and $BDM$ in \cite{arnold2014periodic}).
\begin{proof}
  We first prove that ${\bar \bv}_{\bu} ^1$, if it exists, is orthogonal
  to $\Phi_k$. We denote by $\bu_{\Phi}$
  an element of $\Phi_k$. Then a $\psi^{\Phi}$ exists such that
  $\bu_{\Phi} = \nablaperp \Psi^{\Phi}$. As $\bu_{\Phi} \in \Phi_k$,
  $\Psi^{\Phi}$ is constant on $\partial \hat{K}$. Then
  $$
  \begin{array}{r@{\, = \, }l}
    \Int{\hat{K}}{} {\bar \bv}_{\bu} ^1 \cdot \bu_{\Phi} &
    \Int{\hat{K}}{} {\bar \bv}_{\bu} ^1 \cdot \nablaperp \Psi^{\Phi}  
    \\
    &
    \Int{\hat{K}}{} \left( {\bar \bv}_{\bu} ^1 \right)^\perp \cdot \nabla \Psi^{\Phi}  
    \\
    &
    \Int{\hat{K}}{} \divx \left( \left( {\bar \bv}_{\bu} ^1 \right)^\perp  \Psi^{\Phi}
    \right) 
    \\
    &
    \Int{\partial \hat{K}}{}  \bTr \left(
    \left( {\bar \bv}_{\bu} ^1\right)^\perp  \Psi^{\Phi} \right) 
    \\
    \Int{\hat{K}}{} {\bar \bv}_{\bu} ^1 \cdot \bu_{\Phi}
    & 0,
  \end{array}
  $$
  because both $\left( {\bar \bv}_{\bu} ^1\right)^\perp$ and $\Psi^{\Phi}$
  are constant on $\partial \hat{K}$.
  We have then proven that ${\bar \bv}_{\bu} ^1 \perp \Phi_k$. 
  
  We can now prove the uniqueness of the decomposition. We suppose that
  $$0  = {\bar \bv_{\bu}} ^0 + {\bar \bv_{\bu}} ^1
  + \Sum{f \in \cF (\hat{K})}{} \Sum{i=1}{k}
  {\bar \bv }_{\bu} ^{f,i},
  $$
  where the different components ensure the properties of the proposition. 
  By definition, ${\bar \bv_{\bu}} ^0$ is orthogonal to the
  ${\bar \bv }_{\bu} ^{f,i}$, and we proved that it is also orthogonal
  to ${\bar \bv}_{\bu} ^1$,
  so that it vanishes. We take the trace of the remaining part, which gives
  $$0  =  \bTr \left( {\bar \bv_{\bu}} ^1 \right)
  + \Sum{f \in \cF (\hat{K})}{} \Sum{i=1}{k}
  \bTr \left( {\bar \bv }_{\bu} ^{f,i} \right).
  $$
  Taking the scalar product by any $\mathscr{L}^{g,j}$  for $j \geq 1$ gives
  $$ \Int{\partial \hat{K} }{}
  \bTr \left( {\bar \bv }_{\bu} ^{g,j} \right)  \mathscr{L}^{g,j} = 0.$$
  As $\bTr \left( {\bar \bv }_{\bu} ^{g,j} \right)$ has a moment only on
  \answerI{face $g$ for degree $j$,}
  it is actually zero. As it is orthogonal to
  $\Phi_k$, we get ${\bar \bv }_{\bu} ^{g,j} = 0$ for all $g,j$. It remains then
  ${\bar \bv_{\bu}} ^1 = 0$, and so each component is zero, which proves
  the uniqueness of the decomposition.

  We now prove the existence. We consider one
  $\mathscr{L}^{f,i}$ of $\dd \PP_k \left( \partial \hat{K} \right)$, for $i \geq 1$.
  $\mathscr{L}^{f,i}$ is orthogonal to $\mathbbm{k}$,
  and so using using \autoref{prop:CkTri}, $\mathscr{L}^{f,i}$ is in
  $\Range \bC^{\partial} _k$, so that a   $\psi^{f,i}$ exists such that
  $\bC^{\partial} _k \left( \psi^{f,i} \right) = \mathscr{L}^{f,i}$. Denoting
  by $\mathscr{P}$ the orthogonal projection on $\Phi_k$, we define
  $\be^{f,i} := \nablaperp \left( \psi^{f,i} \right) -
  \mathscr{P} \left( \nablaperp \left( \psi^{f,i} \right) \right)$. Then
  $\be^{f,i}$ is orthogonal to $\Phi_k$, is divergence free, and is such
  that $\bTr\left( \be^{f,i} \right) = \mathscr{L}^{f,i}$. We consider one
  divergence free $\bu \in \bdPP_k (\hat{K})$. We define
  $$\lambda_{f,i} := \Int{\partial \hat{K}}{} \bTr(\bu) \mathscr{L}^{f,i},$$
  and set $ {\bar \bv }_{\bu} ^{f,i} = \lambda_{f,i} \be^{f,i}$. We also
  set ${\bar \bv_{\bu}} ^0 = \mathscr{P} ( \bu )$, and
  $$
  {\bar \bv_{\bu}} ^1 := \bu - {\bar \bv_{\bu}} ^0 -
  \Sum{f \in \cF (\hat{K})}{} \Sum{i=1}{k}
  {\bar \bv }_{\bu} ^{f,i}.
  $$
  Then ${\bar \bv_{\bu}} ^0$ and the ${\bar \bv }_{\bu} ^{f,i}$ ensure all the
  properties required. It remains to prove that ${\bar \bv_{\bu}} ^1$ is
  constant. We know that ${\bar \bv_{\bu}} ^1$ is divergence free, orthogonal
  to $\Phi_k$, and that
  its trace is constant on each face, because all the components
  in $\mathscr{L}^{f,i}$ for $i \geq 1$ were removed.
  We denote by $\bk_y$ the opposite of
  the trace on the side linking $[0,0]$ to $[0,1]$, and $\bk_x$ the opposite of
  the trace on the side linking $[0,0]$ to $[1,0]$, and consider
  $\bk =  (\bk_x,\bk_y)$. As $\divx {\bar \bv_{\bu}} ^1 = 0$, we have
  $$
  \begin{array}{r@{\, = \, }l}
  0 & \Int{\hat{K}}{} \divx {\bar \bv_{\bu}} ^1 \\
   & \Int{\partial \hat{K}}{} \bTr \left( {\bar \bv_{\bu}} ^1 \right) \\
  & \Int{(0,0)}{(1,0)} \bTr \left( {\bar \bv_{\bu}} ^1 \right)
  + \Int{(1,0)}{(1,1)} \bTr \left( {\bar \bv_{\bu}} ^1 \right)
  +\Int{(1,1)}{(0,0)} \bTr \left( {\bar \bv_{\bu}} ^1 \right) \\
  & - \bk_x - \bk_y + \Int{(1,0)}{(1,1)} \bTr \left( {\bar \bv_{\bu}} ^1 \right)
  \\
  0
  & - \bk_x - \bk_y + \sqrt{2}
  \bTr \left( {\bar \bv_{\bu}} ^1 \right)_{|[(1,0),(1,1)]}
  \\
  \end{array}
  ,$$
  which means that the trace on the side linking $[1,0]$ to $[1,1]$ is
  equal to
  $\dfrac{1}{\sqrt{2}} \, \left( \bk_x + \bk_y\right),$
  which is also equal to the trace of $\bk$. This means
  ${\bar \bv_{\bu}} ^1
  - \bk$ is divergence free, and 
  that $\bTr \left({\bar \bv_{\bu}} ^1
  - \bk\right) = 0$, and so
  ${\bar \bv_{\bu}} ^1
  - \bk \in \Phi_k$. 

  As $\bk$ is orthogonal to $\Phi_k$, and so is ${\bar \bv_{\bu}} ^1$, 
  we conclude that ${\bar \bv_{\bu}} ^1
  - \bk$ is also  orthogonal to $\Phi_k$.
  This gives ${\bar \bv_{\bu}} ^1
  - \bk = 0$, and ${\bar \bv_{\bu}} ^1$ is therefore constant. 
\end{proof}

It is important to note that the decomposition of
\autoref{prop:decompositionTri} was proven on the reference element. However,
all the properties of the different spaces are invariant by linear
transformations. This means that the
decomposition of \autoref{prop:decompositionTri} holds actually on any straight
triangular cell of the mesh. 

\subsection{$\ker \left( \ddivx \right)$ and $\Range \left( \nablaperp
  \right)$}

\begin{prop}
  \label{prop:kerRankTri}

  If $\bKK$ denotes the set of uniform vectors, then
  $$ \ker \left( \ddivx \right)
  =
  \Range \left( \nablaperp \right)
  \oplus \bKK.$$
\end{prop}

\begin{proof}
  We begin by proving that $\Range \left( \nablaperp \right) \subset
  \ker \left( \ddivx \right)$. We consider an element $\bu$ of
  $\Range \left( \nablaperp \right)$. Then a $\Psi \in \PP_{k+1}$ exists
  such that $\bu = \nablaperp \psi$. Then
  $$\forall c \in \cC  \qquad \divx \left( \nablaperp \psi \right) = 0.$$
  Also, as $\psi$ is continuous, the jump of $\nablaperp \psi$ across
  faces vanishes. We have then proven that 
  $\Range \left( \nablaperp \right) \subset
  \ker \left( \ddivx \right)$.
  
  We now prove that $\Range \left( \nablaperp \right)
  \perp \bKK$. 
  We denote by $\bk = (\bk_x, \bk_y)^T \in \bKK$, and by $\psi$ an
  element of $\PP_k$. Then
  $$\bk \cdot \nablaperp \psi
  =
  \left(
  \begin{array}{c}
    \bk_x \\
    \bk_y
  \end{array}
  \right)
  \cdot
  \left(
  \begin{array}{c}
    - \partial_y \psi \\
    \partial_x \psi
  \end{array}
  \right)
  = - \bk_x \partial_y \psi + \bk_y \partial_x \psi
  =
  \left(
  \begin{array}{c}
    \bk_y \\
    - \bk_x 
  \end{array}
  \right)
  \cdot \nabla \psi.
  $$
  We denote by $\bkperp = \left(\bk_y , -\bk_x \right)^T$. Then as $\bkperp$
  is uniform, we have on all cells:
  $$\divx \left( \bkperp \psi \right) = \bkperp \cdot \nabla \psi.$$
  Then
  $$
  \begin{array}{r@{\, = \, }l}
    \Sum{c \in \cC}{} \Int{c}{} \bkperp \cdot \nabla \psi &
    \Sum{c \in \cC}{} \Int{c}{} \divx \left( \bkperp \psi \right) \\
    & \Sum{c \in \cC}{} \Int{\partial c}{} \psi \bkperp\cdot \bn_{out} \\
    & \Sum{c \in \cC}{} \Sum{f \in \cF (c)}{}
    \Int{f}{} \psi \bkperp\cdot \bn_{out} \\
    & -\Sum{f}{} \Int{f}{} \Jump{ \psi \bkperp\cdot \bn_f} \\
    \Sum{c \in \cC}{} \Int{c}{} \bkperp \cdot \nabla \psi
    & 0, \\
  \end{array}
  $$
  because both $\bkperp$ and $\psi$ are continuous across the faces. We have
  then proven that
  $$\bKK \perp \Range \left( \nablaperp \right).$$ 
  We also remark that
  $ \bKK \subset \ker \left( \ddivx \right)$, because $\ddivx$ is
  a derivation operator. For the moment, we have proven that
  $$
  \Range \left( \nablaperp \right)
  \oplus \bKK \subset \ker \ddivx.
  $$

  Suppose now that an element $\bu \in \bdPP_k$ is such that
  its divergence is $0$, namely
  $$
  \left\{
  \begin{array}{r}
    \forall c \in \cC \qquad \divx \bu = 0 \\
    \forall f \in \cF \qquad \Jump{\bu \cdot \bn_f } = 0. \\
  \end{array}
  \right.
  $$
  As $\divx \bu = 0$ on all the cells, $\bu$ can be decomposed as
  in \autoref{prop:decompositionTri}; this decomposition involves three types
  of components:
  \begin{itemize}
  \item The ones of $\Phi_k$, which have a trace equal to $0$. This
    set is of dimension $\dfrac{k(k-1)}{2}$ on each cell. 
  \item The constant component; this set is of dimension $2$ on each cell. 
  \item The components ${\bar \bv}_i ^f$ for $1 \leq i \leq k$ and for all
    faces. This set is of dimension $3k$ on each cell. 
  \end{itemize}
  Let us see now what is the effect of the constraint
  $\Jump{\bu \cdot \bn_f} = 0$ on theses different components.
  We first remark that the traces of the
  different components are orthogonal two at a time, which means that we can
  consider the effect of $\Jump{\bu \cdot \bn_f} = 0$ component by component:
  \begin{itemize}
  \item The ones of $\Phi_k$ are not affected by the zero jump constraint,
    because their trace is already 
    equal to $0$. This induces $N \, \dfrac{k(k-1)}{2}$ components
    in $\bdPP_k$. 
  \item The piecewise constant components, with the constraint
    $\Jump{\bu \cdot \bn_f} = 0$ is a set that
    was already identified  in previous
    publications \cite{Dellacherie2016_all_Mach,Arnold_1989,jung2022steady},
    and this space is $\nablaperp \PP_1 \oplus \bKK$, which is of dimension
    $1+\dfrac{N}{2}$.
  \item Concerning the components ${\bar \bv}_i ^f$ for $ 1 \leq i \leq k$,
    the constraint $\Jump{\bu \cdot \bn_f} = 0$ is inducing
    $k \#\cF$ free constraints on a space of dimension $3kN$. This induces a
    space of dimension
    $$3 k N - k \#\cF = 3 k N - k \, \dfrac{3N}{2} =
    \dfrac{3kN}{2}.$$
  \end{itemize}
  Adding the dimension of theses different sets gives the dimension of
  $\ker \ddivx$:
  $$
  \begin{array}{r@{\, = \, }l}
  \dim \left( \ker \ddivx \right)
  & N \, \dfrac{k(k-1)}{2} + 1+\dfrac{N}{2}
  + \dfrac{3kN}{2} \\
  & \dfrac{N}{2}\, \left( k^2 - k + 1 + 3 k\right) + 1 \\
  & \dfrac{N(k+1)^2}{2} + 1, \\
  \end{array}
  $$
  which is exactly equal to $\rank \left( \nablaperp \right) + \dim \bKK$. We have then
  proven that $ \Range \left(  \nablaperp \right) \oplus  \bKK \subset \ker \ddivx$ and that the
  dimensions are equal, so that
  $ \Range \left( \nablaperp \right) \oplus  \bKK = \ker \ddivx$. 
\end{proof}

\subsection{Study of $\ddivx$}

The kernel of $\ddivx$ was already characterised in \autoref{prop:kerRankTri}.
We now characterise its range. 

\begin{prop}[Range of $\ddivx$]
  \label{prop:RankDivxTri}
  We have
  $$
  \dPP_{k-1} \left( \cC \right) \times
  \dPP_{k} \left( \cF \right)
  = \Range \left( \ddivx \right) \oplus \KK,
  $$
  where the sum is orthogonal for the scalar product defined
  in \eqref{eq:scalarProduct}. 
\end{prop}

\begin{proof}
  Following \autoref{prop:kerRankTri} and \autoref{prop:nablaPerpTri} we have
  $$\dim \left( \ker  \ddivx \right) = \rank \left( \nablaperp \right) + 2
  = \dfrac{N (k+1)^2}{2} - 1 + 2 = \dfrac{N (k+1)^2}{2} + 1.$$
  Using the rank nullity theorem gives
  $$
  \dim \left( \bdPP_k \right) = \rank \left( \ddivx \right) +
  \dim \left( \ker \ddivx \right).
  $$
  Using \autoref{prop:dimFETri} leads to
  $$
  \begin{array}{r@{\, = \, }l}
    \rank \left( \ddivx \right) & \dim  \bdPP_k  -
    \dim \left( \ker \ddivx \right) \\
    & N (k+1)(k+2) - \left( \dfrac{N (k+1)^2}{2} + 1 \right) \\
    & \dfrac{N(k+1)(2(k+2) - (k+1))}{2} - 1 \\ 
    & \dfrac{N(k+1)(2k + 4 - k - 1)}{2} - 1 \\
    \rank \left( \ddivx \right)
    & \dfrac{N(k+1)(k + 3)}{2} - 1. \\
  \end{array}
  $$
  We prove now that $\KK$ is orthogonal to
  $\Range \left( \ddivx \right)$. 
  We denote by $k$ an element of $\dPP_{k-1} \left( \cC \right) \times
  \dPP_{k} \left( \cF \right)$, which has the same value on all the cells and
  faces. We also denote by $k$ this value. We denote by $\bu$ an element of
  $\dPP_{k} \left( \cC \right)$. Then
  $$
  \begin{array}{r@{\, = \, }l}
    \scalar{\ddivx \bu}{k}_{[\dPP_{k-1} ( \cC) \times \dPP_k ( \cF)]} &
    \Sum{c \in \cC}{}
    \Int{c}{} k \, \divx \bu
    +
    \Sum{f \in \cF}{}
    \Int{f}{} k \, \Jump{\bu \cdot \bn_f} 
    \\
    &
    \Sum{c \in \cC}{}
    \Int{c}{} \divx \left( k \bu \right)
    +
    \Sum{f \in \cF}{}
    \Int{f}{} k \, \Jump{\bu \cdot \bn_f} 
    \\
    &
    \Sum{c \in \cC}{}
    \Int{\partial c}{} k \bu \cdot \bn_{out}
    +
    \Sum{f \in \cF}{}
    \Int{f}{} k \, \Jump{\bu \cdot \bn_f} 
    \\
    &
    \Sum{c \in \cC}{}
    \Sum{f \in \cF(c)}{}
    \Int{f}{} k \bu \cdot \bn_{out}
    +
    \Sum{f \in \cF}{}
    \Int{f}{} k \, \Jump{\bu \cdot \bn_f} 
    \\
    &
    \Sum{f \in \cF}{}
    \Int{f}{} \left( k \bu_L - k \bu_R  \right)\cdot \bn_{f}
    +
    \Sum{f \in \cF}{}
    \Int{f}{} k \, \Jump{\bu \cdot \bn_f} 
    \\
    &
    -\Sum{f \in \cF}{}
    \Int{f}{} k \Jump{\bu \cdot \bn_{f}}
    +
    \Sum{f \in \cF}{}
    \Int{f}{} k \, \Jump{\bu \cdot \bn_f} 
    \\
    \scalar{\ddivx \bu}{k}_{[\dPP_{k-1} ( \cC) \times \dPP_k ( \cF)]}
    & 0.
  \end{array}
  $$
  We thus have proven that 
  $\Range \left( \ddivx \right) \perp \KK$. As
  $\rank \left( \ddivx \right) =
  \dim \left( \dPP_{k-1} ( \cC) \times \dPP_k ( \cF)\right) - 1$, this actually
  means that
  $$\dPP_{k-1} ( \cC) \times \dPP_k ( \cF) =
  \Range \left( \ddivx \right) \oplus \KK,
  $$
  which ends the proof. 
\end{proof}

\subsection{Summary on the de-Rham complex}

Gathering all the results of this section, the following proposition
was proven
\begin{prop}
\label{prop:SummaryTriangleCurlDiv}
The discrete diagram
$$
\PP_{k+1}
\quad
\diag{\nablaperp}
\quad
\bdPP_k ( \cC )
\quad
\diag{\ddivx}
\quad
\dPP_k ( \cF ) \times \dPP_{k-1} ( \cC),
$$
where 
$\ddivx$ is the $\divx$ in the sense of distributions, 
ensures the \autoref{def:harmonicGap}. Moreover
$$
\left\{
\begin{array}{r@{\, = \, }l}
  \PP_{k+1} / \KK & \ker \left( \nablaperp \right)  \\
  \left( \dPP_k ( \cF ) \times \dPP_{k-1} ( \cC) \right) / \mathbbm{k} &
  \Range \left( \ddivx \right).
\end{array}
\right. 
$$
\end{prop}
By changing the representation of the linear forms, the following
proposition is also obtained:
\begin{prop}
\label{prop:SummaryTriangleGradCurl}
The discrete diagram
$$
\PP_{k+1}
\quad
\diag{\gradx}
\quad
\bdPP_k ( \cC )
\quad
\diag{\dcurlS}
\quad
\dPP_k ( \cF ) \times \dPP_{k-1} ( \cC),
$$
where 
$\dcurlS$ is $\curlS$ in the sense of distributions, 
ensures the \autoref{def:harmonicGap}. Moreover
$$
\left\{
\begin{array}{r@{\, = \, }l}
  \PP_{k+1} / \KK & \ker \left( \gradx \right)  \\
  \left( \dPP_k ( \cF ) \times \dPP_{k-1} ( \cC) \right) / \mathbbm{k} &
  \Range \left( \dcurlS \right).
\end{array}
\right. 
$$
\end{prop}
The location of the degrees of freedom for the two discrete de-Rham
complexes found for triangles are summarised in \autoref{fig:deRhamTriDiv}. 
Note that compared with the finite element spaces
$\dRT_{k+1} ^{\triangle}$ and $\dN_{k+1} ^{\triangle}$
discussed in \autoref{sec:dRT}, the space
$\bdPP_k$ represents a significant improvement regarding the number
of degrees of freedom, as
$$\dim \dRT_{k+1} ^{\triangle} - \bdPP_k
= \dN_{k+1} ^{\triangle} - \bdPP_k
= k+1.$$
Also, Raviart-Thomas and N\'ed\'elec finite element basis are known to be difficult
to generate on simplices, whereas the generation of a basis for
$\bdPP_k$ is straightforward. Therefore, using the basis
$\bdPP_k$ instead of $\dRT_{k+1}$ or $\dN_{k+1}$ for discontinuous approximations
seems to be very beneficial. 

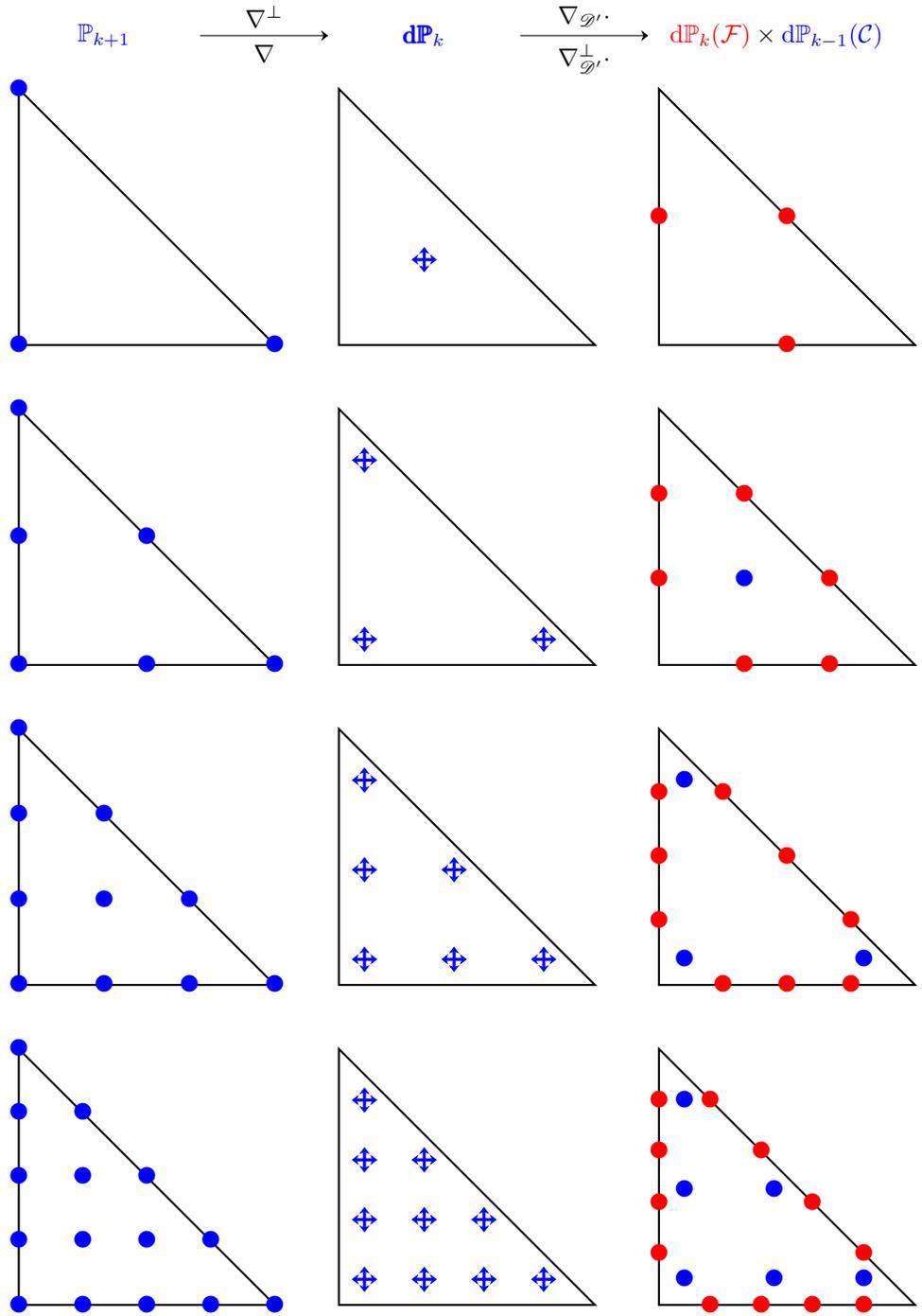
\begin{figure}
  \begin{center}
    \begin{tikzpicture}[scale=0.9]
      \def\LagrangeTriangle#1#2#3#4{
        \draw[-,thick] (#1-#3/3,#2-#3/3) -- (#1+2*#3/3,#2-#3/3) -- (#1-#3/3,#2+2*#3/3) -- cycle;
        \foreach \x in {0,...,#4}{
          \pgfmathsetmacro\maxy{#4-\x};
  
          \foreach \y in {0,...,\maxy}{
            \pgfmathsetmacro\xx{(#1-#3/3)*(1-\x/#4-\y/#4)+(#1+2*#3/3)*\y/#4+(#1-#3/3)*\x/#4};
            \pgfmathsetmacro\yy{(#2-#3/3)*(1-\x/#4-\y/#4)+(#2-#3/3)*\y/#4+(#2+2*#3/3)*\x/#4};
            \node[color=blue] at (\xx,\yy) {\LARGE $\bullet$};
          }
        }
      }

      \node at (2,15.5) {\color{blue} $\PP_{k+1}$};
      \draw[->] (3.5,15.5) -- (5.5,15.5) node [midway, above] {$\nablaperp$} node [midway, below] {$\gradx$};
      \node at (7,15.5) {\color{blue}$\bdPP_{k}$};
      \draw[->] (8.5,15.5) -- (10.5,15.5) node [midway, above] {$\ddivx$} node [midway, below] {$\dcurlS$};
      \node at (12.5,15.5) {$  {\color{red} \dPP_k ( \cF )} \times {\color{blue} \dPP_{k-1} ( \cC)}$};
      \LagrangeTriangle{2}{12}{4}{1};
      \LagrangeTriangle{2}{7}{4}{2};
      \LagrangeTriangle{2}{2}{4}{3};
      \LagrangeTriangle{2}{-3}{4}{4};

      \def\LagrangeVectorDG#1#2#3#4{
        \draw[-,thick] (#1-4/3,#2-4/3) -- (#1+2*4/3,#2-4/3) -- (#1-4/3,#2+2*4/3) -- cycle;
        \foreach \x in {0,...,#4}{
          \pgfmathsetmacro\maxy{#4-\x};
  
          \foreach \y in {0,...,\maxy}{
            \pgfmathsetmacro\xx{(#1-#3/3)*(1-\x/#4-\y/#4)+(#1+2*#3/3)*\y/#4+(#1-#3/3)*\x/#4};
            \pgfmathsetmacro\yy{(#2-#3/3)*(1-\x/#4-\y/#4)+(#2-#3/3)*\y/#4+(#2+2*#3/3)*\x/#4};
            \draw[<->,color=blue,very thick] (\xx-0.2,\yy) -- (\xx+0.2,\yy);
            \draw[<->,color=blue,very thick] (\xx,\yy-0.2) -- (\xx,\yy+0.2);
          }
        }
      }

      \draw[-,thick] (7-4/3,12-4/3) -- (7+2*4/3,12-4/3) -- (7-4/3,12+2*4/3) -- cycle;
      \draw[<->,color=blue,very thick] (7-0.2,12) -- (7+0.2,12);
      \draw[<->,color=blue,very thick] (7,12-0.2) -- (7,12+0.2);
      \LagrangeVectorDG{7}{7}{2.8}{1};
      \LagrangeVectorDG{7}{2}{2.8}{2};
      \LagrangeVectorDG{7}{-3}{2.8}{3};

      \def\dFaceCellsTriangle#1#2#3#4{
        \draw[-,thick] (#1-4/3,#2-4/3) -- (#1+2*4/3,#2-4/3) -- (#1-4/3,#2+2*4/3) -- cycle;
        \ifthenelse{ #4 < 0}
                   {
                   }
                   {
                     \ifthenelse{ #4 > 0}{
                       \foreach \x in {0,...,#4}{
                         \pgfmathsetmacro\maxy{#4-\x};
  
                         \foreach \y in {0,...,\maxy}{
                           \pgfmathsetmacro\xx{(#1-#3/3)*(1-\x/#4-\y/#4)+(#1+2*#3/3)*\y/#4+(#1-#3/3)*\x/#4};
                           \pgfmathsetmacro\yy{(#2-#3/3)*(1-\x/#4-\y/#4)+(#2-#3/3)*\y/#4+(#2+2*#3/3)*\x/#4};
                           \node[color=blue] at (\xx,\yy) {\LARGE $\bullet$};
                         }
                       }
                     }
                     {
                       \node[color=blue] at (#1,#2) {\LARGE $\bullet$};
                     }
                   }

                   \pgfmathsetmacro\maxxa{#4+2};
                   \foreach \xa in {1,...,\maxxa}{
                     \pgfmathsetmacro\xxx{#1-4/3 + 4*\xa/(\maxxa+1)};
                     \pgfmathsetmacro\yyy{#2-4/3 + 4*\xa/(\maxxa+1)};
                     \node[color=red] at (\xxx,#2-4/3) {\LARGE $\bullet$};
                     \node[color=red] at (#1-4/3,\yyy) {\LARGE $\bullet$};
                     \pgfmathsetmacro\xxxx{#1 + (2*4/3) - 4*\xa/(\maxxa+1)};
                     \pgfmathsetmacro\yyyy{#2 -   (4/3) + 4*\xa/(\maxxa+1)};
                     \node[color=red] at (\xxxx,\yyyy) {\LARGE $\bullet$};
                   }
                   
      }
      \dFaceCellsTriangle{12}{12}{2.8}{-1};
      \dFaceCellsTriangle{12}{7}{2.8}{0};
      \dFaceCellsTriangle{12}{2}{2.8}{1};
      \dFaceCellsTriangle{12}{-3}{2.8}{2};

    \end{tikzpicture}
  \end{center}
  \caption{\label{fig:deRhamTriDiv}
    Representation of the finite element spaces involved in
    the grad/curl and curl/div de-Rham complex for triangular meshes
    for $k=0,1,2$ and $3$.
    Same code for colors are used as in \autoref{fig:deRhamQuadDiv}.
  }
\end{figure}

\section{The case of Cartesian meshes}
\label{sec:CartesianMesh}

\subsection{Why the Cartesian case is more complicated}

Inspired by the $\bdPP_0$ triangular case, 
we consider the following discrete divergence
\begin{equation}
  \label{eq:diagQuadFailing}
  \begin{array}{rcr@{\quad \longmapsto \quad }l}
    \ddivx & \, : \, & \bdQQ_0
    & \dPP_0 (\cF) \\
    & & \bu & a \  \text{such that} \  a_f := \Jump{\bu \cdot \bn_f}. 
  \end{array}
\end{equation}
We directly see that $\dim \bdQQ_0 = 2 \# \cC = 2N$ and 
$\dim \dPP_0 (\cF) = \# \cF = 2 N$.
If $N_x$ is the number of cells in the $x$ direction and $N_y$ is the number
of cells in the $y$ direction, then $N = N_x N_y$. Also, it is easy to see
than $\ker \ddivx $ is composed of $N_x + N_y$ components. This means that
$$\rank \ddivx = 2N - N_x - N_y,$$
whereas we expect the $\ddivx$ to be of rank $2N-1$. Also, the
kernel of the divergence is much smaller than what is expected for
correctly approximating continuous divergence free vectors. 

A second problem that we see is that when deriving an element of
$\QQ_k$, it does not give an element of $\QQ_{k-1}$. Therefore, the
finite element space that should be put before the discrete
$\nablaperp$ operator is difficult to determine. If we put
$\QQ_0$, then the derivative will be $0$ and the range of
the discrete $\nablaperp$ will be reduced to $0$. If $\QQ_1$ is used, then
the $\nablaperp$ will be in a space larger than $\QQ_0$.

\subsection{Finite element space definition}

We need to define some finite element spaces for discretising the
different spaces, $A$, $\bB$ and $C$ that are involved in the
de-Rham diagram \eqref{eq:DiagramCurl}. 
Based on what was done in the triangular case \eqref{eq:DiagramTri},
we propose to start by
the continuous finite element space $\QQ_{k+1}$ for the space $A$. For the space
$\bB$, the initial plan is to take $\bdQQ_k$, but we need to enrich it
with the curl of $\QQ_{k+1}$. For simplifying, we relax the
continuity conditions on $\nablaperp \QQ_{k+1}$, and add to $\bdQQ_k$ all the
piecewise discontinuous polynomials that have the same degree as the one
of $\nablaperp \QQ_{k+1}$. This leads to choose $\bdQQkCurl (\cC)$
for discretizing the space $\bB$. Last, taking the divergence in the
sense of distributions for  $\bdQQkCurl (\cC)$
naturally maps to $\dQQ_k ( \cF ) \times \dQQkHat ( \cC)$. This is why
the following diagram is considered
\begin{equation}
  \label{eq:DiagramQuad}
  \QQ_{k+1}
  \quad
  \answerI{\diag{\nablaperp}}
  \quad
  \bdQQkCurl (\cC)
  \quad
  \diag{\ddivx}
  \quad
  \dQQ_k ( \cF ) \times \dQQkHat  (\cC).
\end{equation}

\subsection{Dimension of the finite element spaces}

We first compute the dimension of each of the finite element spaces involved
in \eqref{eq:DiagramQuad}, induced by the number
of faces, points and cells that was computed in \autoref{prop:NElementsQuad}.

\begin{prop}[Dimension of the finite element spaces]
  \label{prop:dimFEQuad}
  $$
  \left\{
  \begin{array}{r@{\, = \, }l}
    \dim \QQ_{k+1} & N(k+1)^2 \\
    \dim \bdQQkCurl ( \cC ) &  N \left( 2(k+1)^2 + 2k +1) \right)\\
    \dim \left( \dQQ_k ( \cF ) \times \dQQkHat ( \cC) \right) &
    N (k^2 + 4 k +2)
    \\
  \end{array}
  \right.
  $$

\end{prop}

\begin{proof}
  The dimension of $\QQ_{k+1}$ was already proven in
  \autoref{prop:nonconformalDimension}.
  We are now interested in $\bdQQkCurl ( \cC )$. Let us recall that
  $$
  \bdQQkCurl (\cC)
  = \left[ \left( \dd \QQ_{k,k} + \dd \QQ_{k-1,k+1} \right) \times
    \left( \dd \QQ_{k,k} + \dd \QQ_{k+1,k-1} \right) \right]
  \oplus \Vect \left(
  \begin{array}{c}
    x^k y^{k+1} \\
    x^{k+1} y^k \\
  \end{array}
  \right). 
  $$
  We first focus on the dimension of
  $\QQ_{k,k} +  \QQ_{k-1,k+1}$. The elements of
  $\QQ_{k-1,k+1}$ are all in $ \QQ_{k,k}$, except for the case in which
  the degree in $y$ is $k+1$. A basis of these polynomials that
  are in $\QQ_{k-1,k+1}$ but not in $ \QQ_{k,k}$ is therefore given
  by $x^i y^{k+1}$ for $0 \leq i \leq k-1$, so that
  $$\dim \left(  \QQ_{k,k} +  \QQ_{k-1,k+1} \right)
  = (k+1)^2 + k.$$
  The space $\QQ_{k,k} + \QQ_{k+1,k-1}$ has the same dimension, so
  that 
  $$
  \begin{array}{r@{\, = \, }l}
    \dim \bdQQkCurl ( \cC ) &  N \left(
    \dim \left( \QQ_{k,k} + \QQ_{k-1,k+1} \right)
    +
    \dim \left( \QQ_{k,k} + \QQ_{k+1,k-1} \right)
    + 1
    \right)\\
    & N \left( 2 \left( (k+1)^2 + k \right) +  1\right) \\
    \dim \bdQQkCurl ( \cC ) 
    & N \left( 2 (k+1)^2 + 2k  +  1\right). 
  \end{array}
  $$
  It remains to compute the dimension of
  $\dQQ_k ( \cF ) \times \dQQkHat ( \cC)$. We recall that
  $$\dQQkHat ( \cC) =
  \dd \QQ_{k-1} ( \cC) + \dd \QQ_{k,k-1} ( \cC) + \dd \QQ_{k-1,k} ( \cC).$$
  This leads to
  $$
  \begin{array}{r@{\, = \, }l}
    \dim \dQQkHat ( \cC) &
    N \left( k^2 + 2k\right)\\
    & N k(k+2).
  \end{array}
  $$
  \answerI{We finally find}
  $$
  \begin{array}{r@{\, = \, }l}
    \dim \left( \dQQ_k ( \cF ) \times \dQQkHat ( \cC) \right)
    & \dim \left( \dQQ_k ( \cF ) \right) + \dim \dQQkHat ( \cC) 
    \\
    & \# \cF (k+1) + N k (k+2) \\
    & 2 N (k+1) + N k(k+2) \\
    \dim \left( \dQQ_k ( \cF ) \times \dQQkHat ( \cC) \right)
    & N (k^2 + 4 k + 2). \\
  \end{array}
  $$
\end{proof}

\subsection{Study of $\dnablaperp$}

We are now interested in the study of the $\dnablaperp$ operator.
\begin{prop}[$\dnablaperp$ in the quadrangular case]
  \label{prop:nablaPerpQuad}
  We have
  $$
  \left\{
  \begin{array}{r@{\, = \, }l}
    \dim \left( \ker \dnablaperp \right) & 1 \\
    \rank \left( \dnablaperp \right)     & N (k+1) ^2 - 1. \\
  \end{array}
  \right.
  $$
\end{prop}
The proof is exactly the same as \autoref{prop:nablaPerpTri}.


\subsection{Discrete divergence free polynomials on the reference cell}

We consider the following application
\begin{equation}
  \bC ^{\partial} _k \, : \, \psi \in \dQQ_{k+1} (\hat{K}) \, \longmapsto \,
  \bTr \left( \nablaperp \psi \right) \in \dd \PP_k ( \partial \hat{K} )
\end{equation}

\begin{prop}
  \label{prop:CkQuad}
  We denote by $\mathbbm{k}$ the constant elements of
  $\dd \PP_k ( \partial \hat{K} )$. Then 
  $$\dd \PP_k ( \partial \hat{K} ) = \Range \bC^{\partial} _k \oplus \mathbbm{k},$$
  where the sum is orthogonal. 
\end{prop}

\begin{proof}
  The proof that the sum of $\Range \bC^{\partial} _k $ and $\mathbbm{k}$ is
  direct and orthogonal follows exactly the same lines as the proof
  for the triangular case of \autoref{prop:CkTri}.

  We are now interested in the study of the kernel of $\bC^{\partial} _k$.
  Suppose that an element $\psi$ is such that $\bC^{\partial} _k (\psi) = 0$.
  We consider the classical Lagrange basis of $\dQQ_{k+1} (\hat{K})$.
  Then $\psi$ is such that its value on the boundary
  of $\hat{K}$ is constant, and may take any value on the degrees of freedom
  matching with the interior nodes. This means that
  $$\dim \left( \ker \bC^{\partial} _k \right) = 1 + k^2. $$
  Using the rank-nullity theorem gives
  $$
  \begin{array}{r@{\, = \, }l}
  \rank \left( \bC^{\partial} _k \right)
  &
  \dim \dQQ_{k+1} - \left( 1 + k^2 \right) \\
  &
  (k+2)^2 -  1 - k^2  \\
  &
  k^2 + 4 k + 4 - 1 - k^2 
  \\
  \rank \left( \bC^{\partial} _k \right)
  & 4 k +3. \\
  \end{array}
  $$
  We also know that $\dim \dd \PP_k \left( \partial \hat{K} \right) = 4(k+1)$.
  We have then $\Range \bC^{\partial} _k \oplus \mathbbm{k}
  \subset \dd \PP_k \left( \partial \hat{K} \right)$, and
  $\dim \left( \Range \bC^{\partial} _k \oplus \mathbbm{k} \right)  =
  \dim \dd \PP_k \left( \partial \hat{K} \right)$, so that
  $\Range \bC^{\partial} _k \oplus \mathbbm{k}  =
  \dd \PP_k \left( \partial \hat{K} \right)$, which ends the proof. 

\end{proof}

\begin{prop}[Decomposition of divergence free elements]
  \label{prop:decompositionQuad}
  We denote by $\mathscr{L}^{f,i}$  the Legendre polynomial of
  degree $i$
  on the face $f$ of $\hat{K}$, normalised such that
  $$\Int{\partial \hat{K}}{} \left( \mathscr{L}^{f,i} \right)^2 = 1.$$

  Suppose that $\bu \in \bdQQkCurl (\hat{K})$ is divergence free. Then
  $\bu$ can be uniquely decomposed as
  $$\bu = {\bar \bv_{\bu}} ^0 + {\bar \bv_{\bu}} ^1
  + \Sum{f \in \cF (\hat{K})}{} \Sum{i=1}{k}
  {\bar \bv }_{\bu} ^{f,i},
  $$
  where
  \begin{itemize}
  \item $\answerI{\bar \bv}_{\bu} ^0$ is in the set
    $$\Phi_k = \left\{
    \bu \in  \bdQQkCurl \qquad \divx \bu = 0 \qquad \bTr(\bu) = 0
    \right\}.$$
  \item ${\bar \bv}_{\bu} ^1$ is in
    $\mathrm{Vec} \left(
    \left(
    \begin{array}{c}
      1 \\
      0
    \end{array}
    \right), 
    \left(
    \begin{array}{c}
      0 \\
      1
    \end{array}
    \right), 
    \left(
    \begin{array}{c}
      1-2x \\
      2y-1
    \end{array}
    \right)
    \right)$.
  \item ${\bar \bv }_{\bu} ^{f,i}$ is orthogonal to $\Phi_k$ and such that
    \begin{itemize}
    \item $ \divx {\bar \bv }_{\bu} ^{f,i} = 0$. 
    \item $\bTr \left( {\bar \bv }_{\bu} ^{f,i} \right)$ is orthogonal to all
      $\mathscr{L}^{g,j}$ for $\left\{g,j\right\} \neq \left\{f,i\right\}$.
    \end{itemize}
  \end{itemize}
\end{prop}

\begin{proof}
  We first prove that ${\bar \bv}_{\bu} ^1 \perp \Phi_k$.
  We denote by $\bu_{\Phi}$ an element of $\Phi_k$. 
  As in the proof of \autoref{prop:decompositionTri}, a
  $\psi^{\Phi}$ exists such that
  $\bu_{\Phi} = \nablaperp \Psi^{\Phi}$, and the trace of 
  $\Psi^{\Phi}$ is constant. Then
  $$
  \begin{array}{r@{\, = \, }l}
    \Int{\hat{K}}{} {\bar \bv}_{\bu} ^1 \cdot \bu_{\Phi} &
    \Int{\hat{K}}{} {\bar \bv}_{\bu} ^1 \cdot \nablaperp \Psi^{\Phi}  
    \\
    &
    \Int{\hat{K}}{} \left( {\bar \bv}_{\bu} ^1 \right)^\perp \cdot \nabla \Psi^{\Phi}  
    \\
    &
    \Int{\hat{K}}{} \divx \left( \left( {\bar \bv}_{\bu} ^1 \right)^\perp  \Psi^{\Phi}
    \right)
    -
    \Int{\hat{K}}{} \divx \left( \left( {\bar \bv}_{\bu} ^1 \right)^\perp
    \right) \Psi^{\Phi}
    \\
    &
    \Int{\partial \hat{K}}{}  \bTr \left(
    \left( {\bar \bv}_{\bu} ^1\right)^\perp  \Psi^{\Phi} \right) 
    \\
    \Int{\hat{K}}{} {\bar \bv}_{\bu} ^1 \cdot \bu_{\Phi}
    &
    \Tr \left( \Psi^{\Phi} \right)
    \Int{\partial \hat{K}}{}  \bTr \left(
    \left( {\bar \bv}_{\bu} ^1\right)^\perp \right) 
    ,
  \end{array}
  $$
  because $\divx \left( \left( {\bar \bv}_{\bu} ^1 \right)^\perp \right) = 0$, 
  and because $\Tr \left( \Psi^{\Phi} \right)$ is constant. Computing
  the integral on $\partial \hat{K}$ of all the components of
  $\left( {\bar \bv}_{\bu} ^1\right)^\perp$
  (namely $(1,0)^T$, $(0,1)^T$ and
  $\left(1 - 2x, 2y - 1 \right)^T$ ) leads to
  $$ \Int{\hat{K}}{} {\bar \bv}_{\bu} ^1 \cdot \bu_{\Phi} = 0. $$
  We have then proven that ${\bar \bv}_{\bu} ^1 \perp \Phi_k$. 
  The proof of uniqueness follows exactly the same lines as the proof
  of \autoref{prop:decompositionTri}.

  The strategy for the existence begins
  similarly but then differs. 
  We consider one $\mathscr{L}^{f,i}$ of
  $\dd \PP_k \left( \partial \hat{K} \right)$, for $i \geq 1$. Then, as was done
  in the proof of \autoref{prop:decompositionTri}, a
  $\psi^{f,i}$ exists such that
  $\bC^{\partial} _k \left( \psi^{f,i} \right) = \mathscr{L}^{f,i}$. Denoting
  by $\mathscr{P}$ the orthogonal projection on $\Phi_k$, we define
  $\be^{f,i} := \nablaperp \left( \psi^{f,i} \right) -
  \mathscr{P} \left( \nablaperp \left( \psi^{f,i} \right) \right)$. Then
  $\be^{f,i}$ is orthogonal to $\Phi_k$, is divergence free, and is such
  that $\bTr\left( \be^{f,i} \right) = \mathscr{L}^{f,i}$. We consider now one
  divergence free $\bu \in  \bdQQkCurl (\hat{K})$. We define
  $$\lambda_{f,i} := \Int{\partial \hat{K}}{} \bTr(\bu) \mathscr{L}^{f,i},$$
  and set $ {\bar \bv }_{\bu} ^{f,i} = \lambda_{f,i} \be^{f,i}$. We also
  set ${\bar \bv_{\bu}} ^0 = \mathscr{P} ( \bu )$, and
  $$
  {\bar \bv_{\bu}} ^1 := \bu - {\bar \bv_{\bu}} ^0 -
  \Sum{f \in \cF (\hat{K})}{} \Sum{i=1}{k}
  {\bar \bv }_{\bu} ^{f,i}.
  $$
  Then ${\bar \bv_{\bu}} ^0$ and the ${\bar \bv }_{\bu} ^{f,i}$ ensure all the
  properties required. It remains to prove that ${\bar \bv_{\bu}} ^1$ is
  in the expected space.
  We know that ${\bar \bv_{\bu}} ^1$ is divergence free, orthogonal
  to $\Phi_k$, and that
  its trace is constant on each face, because all the components
  in $\mathscr{L}^{f,i}$ for $i \geq 1$ were removed. 
  We denote by $b_0$,$b_1$,$b_2$ and $b_3$ the
  values of the trace on the boundary of ${\bar \bv_{\bu}} ^1$
  (see \autoref{fig:dofBoundary}).
  \begin{figure}
    \begin{center}
      \begin{tikzpicture}
        \node[draw, minimum size=2cm, regular polygon, 
          regular polygon sides=4, 
          label=side 1:$0$, label=side 2:$-1$, label=side 3:$0$, label=side 4:$1$] (A) {};
        \node[below = 10mm of A] (A2)
             {$\left(
               \begin{array}{c}
                 1 \\
                 0
               \end{array}
               \right)$};

        \node[right = 15mm of A,draw, minimum size=2cm, regular polygon, 
          regular polygon sides=4, 
          label=side 1:$1$, label=side 2:$0$, label=side 3:$-1$, label=side 4:$0$] (B) {};
        \node[below = 10mm of B] (B2)
             {$\left(
               \begin{array}{c}
                 0 \\
                 1
               \end{array}
               \right)$};

        \node[right = 15mm of B,draw, minimum size=2cm, regular polygon, 
          regular polygon sides=4, 
          label=side 1:$1$, label=side 2:$-1$, label=side 3:$1$, label=side 4:$-1$] (C) {};
        \node[below = 10mm of C] (C2)
             {$\left(
               \begin{array}{c}
                 1-2x \\
                 2y - 1
               \end{array}
               \right)$};

        \node[right = 15mm of C,draw, minimum size=2cm, regular polygon, 
          regular polygon sides=4, 
          label=side 1:$b_2$, label=side 2:$b_3$, label=side 3:$b_0$, label=side 4:$b_1$] (D) {};
        \node[below = 10mm of D] (D2)
             {${\bar \bv_{\bu}} ^1$};

      \end{tikzpicture}
      \caption{\label{fig:dofBoundary} Value of the trace for
        ${\bar \bv_{\bu}} ^1$ and the basis used for representing it.}
    \end{center}
  \end{figure}
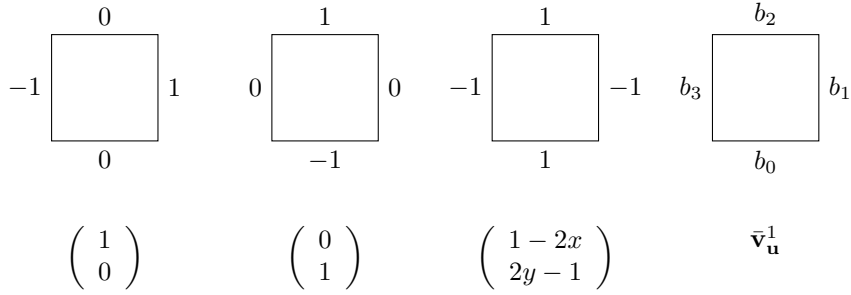
  We then define
  $$
  \bk :=
  \dfrac{b_1-b_3}{2}\, 
  \left(
  \begin{array}{c}
    1 \\
    0 \\
  \end{array}
  \right)
  +
  \dfrac{b_2-b_0}{2}\, 
  \left(
  \begin{array}{c}
    0 \\
    1 \\
  \end{array}
  \right)
  +
  \dfrac{b_0+b_2}{2}\, 
  \left(
  \begin{array}{c}
    1 - 2x \\
    2y - 1  \\
  \end{array}
  \right)
  ,
  $$
  then $\Tr\left(\bk\right) = \Tr\left({\bar \bv_{\bu}} ^1\right)$ (note that
  as $\divx {\bar \bv_{\bu}} ^1 = 0$, we have $b_0 + b_1 + b_2 + b_3 = 0$).
  Also, $\bk$ is divergence free and orthogonal to $\Phi_k$. This
  means that ${\bar \bv_{\bu}} ^1 - \bk$ is divergence free and orthogonal
  to $\Phi_k$. But as $\Tr \left( {\bar \bv_{\bu}} ^1 - \bk \right) = 0$,
  ${\bar \bv_{\bu}} ^1 - \bk$ is also in $\Phi_k$, and so 
  ${\bar \bv_{\bu}} ^1 - \bk = 0$, which proves that
  ${\bar \bv_{\bu}} ^1$ is in
  $\mathrm{Vec} \left(
  \left(
  \begin{array}{c}
    1 \\
    0
  \end{array}
  \right), 
  \left(
  \begin{array}{c}
    0 \\
    1
  \end{array}
  \right), 
  \left(
  \begin{array}{c}
    1-2x \\
    2y-1
  \end{array}
  \right)
  \right)$. 
  This ends the proof. 

\end{proof}

It is important to note that the decomposition of
\autoref{prop:decompositionQuad} was proven on the reference element,
but holds also on all the cells of a Cartesian meshes, the basis function
$\left(
\begin{array}{c}
  1-2x \\
  2y-1
\end{array}
\right)
$
being replaced by
$
\left(
\begin{array}{c}
  \dfrac{2}{L_x} \, \left( m_x - x \right) \\
  \dfrac{2}{L_y} \, \left( y - m_y \right)
\end{array}
\right)
$,
where $(m_x, m_y)$ is the centre of the cell and $L_x$ and $L_y$ are
the size of the cell in the directions $x$, and $y$. 

\subsection{$\ker \left( \ddivx \right)$ and $\Range \left( \nablaperp
  \right)$}

\begin{prop}
  \label{prop:kerRankQuad}

  We denote by $\bKK$ the set of uniform vectors. Then
  $$ \ker \left( \ddivx \right)
  =
  \Range \left( \nablaperp \right)
  \oplus \bKK.$$
\end{prop}

\begin{proof}
  The proof of $\Range \left( \nablaperp \right)$ and 
  $\bKK$ being orthogonal and the two being subvectorial spaces of
  $ \ker \left( \ddivx \right)$ is exactly the same as in the proof of
  \autoref{prop:kerRankTri}.

  Concerning the dimension of this space, we proceed as in
  the proof of \autoref{prop:kerRankTri}, but need to rewrite it because the
  dimensions are different.
  Suppose now that an element $\bu \in \bdQQkCurl$ is such that
  its divergence $\ddivx$ is $0$, namely
  $$
  \left\{
  \begin{array}{r}
    \forall c \in \cC \qquad \divx \bu = 0 \\
    \forall f \in \cF \qquad \Jump{\bu \cdot \bn_f } = 0. \\
  \end{array}
  \right.
  $$
  As $\divx \bu = 0$ on all the cells, $\bu$ can be decomposed as
  in \autoref{prop:decompositionQuad}; this decomposition involves three types
  of components:
  \begin{itemize}
  \item The ones of $\Phi_k$, which have a trace equal to $0$. This
    set is of dimension $k^2 +1$ on each cell. 
  \item One component in the dimension 3 vectorial space.
    $$\widehat{\bKK} = \mathrm{Vec} \left(
    \left(
    \begin{array}{c}
      1 \\
      0
    \end{array}
    \right)
    ,
    \left(
    \begin{array}{c}
      0 \\
      1
    \end{array}
    \right)
    ,
    \left(
    \begin{array}{c}
      \dfrac{2}{L_{i,j} ^x} \, \left( m_{i,j} ^x - x \right) \\
      \dfrac{2}{L_{i,j} ^y} \, \left( y - m_{i,j} ^y \right)
    \end{array}
    \right)
    \right).$$
  \item The components ${\bar \bv}_i ^f$ for $1 \leq i \leq k$ and for all
    faces. This set is of dimension $4k$ on each cell. 
  \end{itemize}
  Let us see now what is the effect of the constraint
  $\Jump{\bu \cdot \bn_f} = 0$ on theses different components.
  We first remark that the traces of the
  different components are orthogonal two at a time, which means that we can
  consider the effect of $\Jump{\bu \cdot \bn_f} = 0$ component by component:
  \begin{itemize}
  \item The ones of $\Phi_k$ are not affected by the zero jump constraint,
    because their trace is already 
    equal to $0$. This induces $N \,  k^2 $ components
    in $\bdQQkCurl$. 
  \item The components of $\widehat{\bKK}$, with the constraint
    $\Jump{\bu \cdot \bn_f} = 0$ is a set that is identified in
    \autoref{quad_lowOrder}, proven in
    \autoref{app:Quad}: it is of dimension
    $N+1$.
  \item Concerning the components ${\bar \bv}_i ^f$ for $ 1 \leq i \leq k$,
    the constraint $\Jump{\bu \cdot \bn_f} = 0$ is inducing
    $k \#\cF$ free constraints on a space of dimension $4kN$. This induces a
    space of dimension
    $$4 k N - k \#\cF = 4 k N - k 2 N = 2 k N.$$
  \end{itemize}
  Adding the dimension of theses different sets gives the dimension of
  $\ker \ddivx$:
  $$
  \begin{array}{r@{\, = \, }l}
  \dim \left( \ker \ddivx \right)
  & N k^2  + N+1 + 2 k N 
  \\
  & N \left( k^2 + 2 k + 1 \right) + 1 \\
  & N \left( k+1 \right)^2 + 1, \\
  \end{array}
  $$
  which is exactly equal to $\rank \left( \nablaperp \right) + \dim \bKK$. We have then
  proven that $ \nablaperp \oplus  \bKK \subset \ker \ddivx$ and that the
  dimensions are equal, so that
  $ \Range \left( \nablaperp \right) \oplus  \bKK = \ker \ddivx$.

\end{proof}

\subsection{Study of $\ddivx$}

The kernel of $\ddivx$ was already characterised in \autoref{prop:kerRankQuad}.
We now characterise its range. 

\begin{prop}[Range of $\ddivx$]
  \label{prop:RankDivxQuad}
  We have
  $$
  \dQQkHat ( \cC)
  \times
  \dQQ_{k} \left( \cF \right)
  = \Range \left( \ddivx \right) \oplus \KK,
  $$
  where the sum is orthogonal for the scalar product defined
  in \eqref{eq:scalarProduct}. 
\end{prop}

The proof is exactly the same as the proof of \autoref{prop:RankDivxTri}.

\subsection{Summary on the de-Rham complex}

Gathering all the results of this section, the following proposition
was proven
\begin{prop}
\label{prop:SummaryQuadCurlDiv}
The discrete diagram
$$
\PP_{k+1}
\quad
\diag{\nablaperp}
\quad
\bdQQkCurl ( \cC )
\quad
\diag{\ddivx}
\quad
\dQQ_k ( \cF ) \times \dQQkHat ( \cC),
$$
where 
$\ddivx$ is $\divx$ in the sense of distributions, 
ensures the  \autoref{def:harmonicGap}. Moreover
$$
\left\{
\begin{array}{r@{\, = \, }l}
  \QQ_{k+1} / \KK & \ker \left( \dnablaperp \right)  \\
  \left( \dQQ_k ( \cF ) \times \dQQkHat ( \cC) \right) / \mathbbm{k} &
  \Range \left( \ddivx \right).
\end{array}
\right. 
$$
\end{prop}

By changing the representation of the linear forms, which is equivalent
to rotating of $\pi/2$ the vector spaces, the following
proposition is also obtained:
\begin{prop}
\label{prop:SummaryQuadrangleGradCurl}
The discrete diagram
$$
\QQ_{k+1}
\quad
\diag{\gradx}
\quad
\bdQQkDiv ( \cC )
\quad
\diag{\dcurlS}
\quad
\dQQ_k ( \cF ) \times \dQQkHat ( \cC),
$$
where 
$\dcurlS$ is $\curlS$ in the sense of distributions, 
ensures the \autoref{def:harmonicGap}. Moreover
$$
\left\{
\begin{array}{r@{\, = \, }l}
  \QQ_{k+1} / \KK & \ker \left( \gradx \right)  \\
  \left( \dQQ_k ( \cF ) \times \dQQkHat ( \cC) \right) / \mathbbm{k} &
  \Range \left( \dcurlS \right).
\end{array}
\right. 
$$
\end{prop}

We now compare the number of degrees of freedom of the
basis $\bdQQkDiv$ and $\bdQQkCurl$ with the quadrangular
basis discussed in \autoref{sec:dRT} on each cell:
$$
\dim \bdQQkDiv - \dim \dRT_{k+1} ^{\square}
=
\dim \bdQQkCurl - \dim \dN_{k+1} ^{\square}
= 2(k+2)(k+1) - \left( 2 (k+1)^2 + 2k+1 \right)
= 1.
$$
Therefore the difference of number of degrees of freedom is negligible.
Considering that the vector basis $\bdQQkDiv$ and
$\bdQQkCurl$ do not have a Lagrange basis (this is why no representation
of these basis was proposed), it seems that using these new basis has few
benefits with respect to $\dRT_{k+1}$ and $\dN_{k+1}$.

\section{Conclusion}
\label{sec:Conclusion}

In this article, two-dimensional discrete de-Rham structures in which
the vector space is a discontinuous approximation space were discussed.
We first recalled that by relaxing the normal or tangential continuity
properties of the classical conformal space, a set of discontinuous
approximation space can be designed as in \cite{licht2017complexes}.
These discontinuous spaces, $\dN$ and $\dRT$, are discontinuous
versions of the N\'ed\'elec $\bN$ and Raviart-Thomas $\RT$
approximation spaces. 

Then the de-Rham structure of the natural discontinuous vectorial
space $\bdPP_k$ on triangles, used for example
for the discontinuous Galerkin method was investigated. We proved that for
straight triangular meshes, a discrete de-Rham complex can be built for which
the \autoref{def:harmonicGap} is ensured for any order of approximation. 

Based on the finite element spaces and discrete
$\nablaperp$/$\left( \ddivx \right)$
or $\gradx$/$\left( \dcurlS\right)$
that were used for the triangular case, \autoref{def:harmonicGap}
was proven for discontinuous spaces of
vectors. However, the space of vectors is not the classical
$\bdQQ_k$ approximation space which is usually used in the
discontinuous Galerkin method, but rather an enriched version, 
$\bdQQkCurl$ and $\bdQQkDiv$, depending on the diagram considered. Note that
no diagram that would be based on the so-called serendipity elements was
addressed, but this could be a way to derive velocity approximation
spaces that can be put in a de-Rham diagram with fewer degrees of freedom 
(still at the price of getting a nonoptimal order of
accuracy on general quads). 

It is important to note that only \autoref{def:harmonicGap} was
addressed in this article. \answerI{The bounded cochain projection property}
was not addressed, which is still far from
the framework that was developed in
\cite{arnold2018finite} 
for conforming finite element approximation. 
Still, \autoref{def:harmonicGap} is an algebraic property that we
believe to be useful in the context of the derivation of curl
preserving numerical schemes for hyperbolic systems. Some
curl-preserving schemes that were developed in the finite volume
scheme context 
\cite{Dellacherie2010_cell_geometry,Dellacherie2016_all_Mach,jung2022steady}
rely on the existence of the following discrete decomposition
\cite{Arnold_1989}
($\mathbb{CR}$ denotes the Crouzeix-Raviart finite element space
\cite{crouzeix1973conforming}), which reads on periodic triangular meshes as
\begin{equation}
  \label{eq:HHD_LO}
  \bdPP_0 / \R^2 = \nablaperp \PP_1 \oplus \nabla \mathbb{CR},
\end{equation}
and on the preservation of the solenoidal component, this property being
strongly linked with the correct low Mach number behaviour on triangular
and tetrahedral meshes \cite{guillard2009behavior,jung2024behavior}.
The \autoref{def:harmonicGap} discussed in this article 
directly induces the following Hodge-Helmholtz decomposition
\cite{arnold2018finite}:
\begin{equation}
  \label{eq:HHD}
  \bB / \R^2=
  \Range \left( \nablaperp \right)
  \oplus
  \Range \left( \ddivx ^{\star} \right), 
\end{equation}
where the $\star$ denotes the adjoint operator. For example,
the diagram of \autoref{prop:SummaryTriangleCurlDiv} induces the following
discrete Hodge-Helmholtz decomposition
\begin{equation}
  \label{eq:HHD_HO}
  \bdPP_k ( \cC ) / \R^2=
  \Range \left( \nablaperp \PP_{k+1} \right)
  \oplus
  \Range \left( \ddivx^{\star} \left(
  \dPP_k ( \cF ) \times \dPP_{k-1} ( \cC)
  \right) \right), 
\end{equation}
which can be seen as the high order extension of \eqref{eq:HHD_LO}.
\answerI{
  In a submitted article
  \cite{perrier2024developmentdiscontinuousgalerkinmethods},
  we explain how to preserve a curl or a divergence
  constraint of a hyperbolic system of conservation law with the approximation
  spaces that were discussed in this article.
}

\bibliographystyle{plain}
\bibliography{biblio}

\begin{thebibliography}{10}

\bibitem{arnold2018finite}
Douglas~Norman Arnold.
\newblock {\em Finite element exterior calculus}.
\newblock SIAM, 2018.

\bibitem{arnold2005quadrilateral}
Douglas~Norman Arnold, Daniele Boffi, and Richard~Steven Falk.
\newblock Quadrilateral {H} (div) finite elements.
\newblock {\em SIAM Journal on Numerical Analysis}, 42(6):2429--2451, 2005.

\bibitem{Arnold_1989}
Douglas~Norman Arnold and Richard~Steven Falk.
\newblock A uniformly accurate finite element method for the
  {R}eissner-{M}indlin plate.
\newblock {\em SIAM J. Numer. Anal.}, 26:1276--1290, 1989.

\bibitem{arnold2006finite}
Douglas~Norman Arnold, Richard~Steven Falk, and Ragnar Winther.
\newblock Finite element exterior calculus, homological techniques, and
  applications.
\newblock {\em Acta numerica}, 15:1--155, 2006.

\bibitem{arnold2010finite}
Douglas~Norman Arnold, Richard~Steven Falk, and Ragnar Winther.
\newblock Finite element exterior calculus: from {H}odge theory to numerical
  stability.
\newblock {\em Bulletin of the American mathematical society}, 47(2):281--354,
  2010.

\bibitem{arnold2014periodic}
Douglas~Norman Arnold and Anders~Bernhard Logg.
\newblock Periodic table of the finite elements.
\newblock {\em SIAM News}, 47(9):212, 2014.

\bibitem{balsara2001divergence}
Dinshaw~S. Balsara.
\newblock Divergence-free adaptive mesh refinement for magnetohydrodynamics.
\newblock {\em Journal of Computational Physics}, 174(2):614--648, 2001.

\bibitem{balsara2004second}
Dinshaw~S. Balsara.
\newblock Second-order-accurate schemes for magnetohydrodynamics with
  divergence-free reconstruction.
\newblock {\em The Astrophysical Journal Supplement Series}, 151(1):149, 2004.

\bibitem{balsara1999staggered}
Dinshaw~S. Balsara and Daniel~Shields Spicer.
\newblock A staggered mesh algorithm using high order {G}odunov fluxes to
  ensure solenoidal magnetic fields in magnetohydrodynamic simulations.
\newblock {\em Journal of Computational Physics}, 149(2):270--292, 1999.

\bibitem{bonelle2014compatible}
J{\'e}r{\^o}me Bonelle.
\newblock Compatible discrete operator schemes on polyhedral meshes for
  elliptic and {S}tokes equations.
\newblock {\em Ph. D. Thesis}, 2014.

\bibitem{bonelle2014analysis}
J{\'e}r{\^o}me Bonelle and Alexandre Ern.
\newblock Analysis of compatible discrete operator schemes for elliptic
  problems on polyhedral meshes.
\newblock {\em ESAIM: Mathematical Modelling and Numerical Analysis},
  48(2):553--581, 2014.

\bibitem{bonelle2015analysis}
J{\'e}r{\^o}me Bonelle and Alexandre Ern.
\newblock Analysis of compatible discrete operator schemes for the {S}tokes
  equations on polyhedral meshes.
\newblock {\em IMA Journal of numerical analysis}, 35(4):1672--1697, 2015.

\bibitem{bossavit1988whitney}
Alain Bossavit.
\newblock Whitney forms: A class of finite elements for three-dimensional
  computations in electromagnetism.
\newblock {\em IEE Proceedings A (Physical Science, Measurement and
  Instrumentation, Management and Education, Reviews)}, 135(8):493--500, 1988.

\bibitem{bossavit1998computational}
Alain Bossavit.
\newblock {\em Computational electromagnetism: variational formulations,
  complementarity, edge elements}.
\newblock Academic Press, 1998.

\bibitem{brezzi1985two}
Franco Brezzi, Jim Douglas, and Luisa~Donatella Marini.
\newblock Two families of mixed finite elements for second order elliptic
  problems.
\newblock {\em Numerische Mathematik}, 47(2):217--235, 1985.

\bibitem{brezzi2012mixed}
Franco Brezzi and Michel Fortin.
\newblock {\em Mixed and hybrid finite element methods}, volume~15.
\newblock Springer Science \& Business Media, 2012.

\bibitem{crouzeix1973conforming}
Michel Crouzeix and Pierre-Arnaud Raviart.
\newblock Conforming and nonconforming finite element methods for solving the
  stationary {S}tokes equations {I}.
\newblock {\em Revue fran{\c{c}}aise d'automatique informatique recherche
  op{\'e}rationnelle. Math{\'e}matique}, 7(R3):33--75, 1973.

\bibitem{Dellacherie2016_all_Mach}
St{\'e}phane Dellacherie, Jonathan Jung, Pascal Omnes, and Pierre-Arnaud
  Raviart.
\newblock {Construction of modified {G}odunov type schemes accurate at any
  {M}ach number for the compressible {E}uler system}.
\newblock {\em {Mathematical Models and Methods in Applied Sciences}}, November
  2016.

\bibitem{Dellacherie2010_cell_geometry}
St{\'e}phane Dellacherie, Pascal Omnes, and Felix Rieper.
\newblock The influence of cell geometry on the \mbox{G}odunov scheme applied
  to the linear wave equation.
\newblock {\em Journal of Computational Physics}, 229(14):5315--5338, 2010.

\bibitem{di2020hybrid}
Daniele~Antonio Di~Pietro and J{\'e}r{\^o}me Droniou.
\newblock {\em The {H}ybrid {H}igh-{O}rder method for polytopal meshes}.
\newblock Number~19 in Modeling, Simulation and Applications series. {Springer
  International Publishing}, March 2020.

\bibitem{ern2021finite}
Alexandre Ern and Jean-Luc Guermond.
\newblock {\em Finite elements I: Approximation and interpolation}, volume~72.
\newblock Springer Nature, 2021.

\bibitem{guillard2009behavior}
Herv{\'e} Guillard.
\newblock On the behavior of upwind schemes in the low {M}ach number limit.
  {IV}: P0 approximation on triangular and tetrahedral cells.
\newblock {\em Computers \& Fluids}, 38(10):1969--1972, 2009.

\bibitem{hiptmair2001discrete}
Ralf Hiptmair.
\newblock Discrete {H}odge operators.
\newblock {\em Numerische Mathematik}, 90:265--289, 2001.

\bibitem{hiptmair2002finite}
Ralf Hiptmair.
\newblock Finite elements in computational electromagnetism.
\newblock {\em Acta Numerica}, 11:237--339, 2002.

\bibitem{hong2022extended}
Qingguo Hong, Yuwen Li, and Jinchao Xu.
\newblock An extended {G}alerkin analysis in finite element exterior calculus.
\newblock {\em Mathematics of Computation}, 91(335):1077--1106, 2022.

\bibitem{jung2022steady}
Jonathan Jung and Vincent Perrier.
\newblock Steady low {M}ach number flows: identification of the spurious mode
  and filtering method.
\newblock {\em Journal of Computational Physics}, 468:111462, 2022.

\bibitem{jung2024behavior}
Jonathan Jung and Vincent Perrier.
\newblock Behavior of the discontinuous {G}alerkin method for compressible
  flows at low {M}ach number on triangles and tetrahedrons.
\newblock {\em SIAM Journal on Scientific Computing}, 46(1):A452--A482, 2024.

\bibitem{JUNG2024113252}
Jonathan Jung and Vincent Perrier.
\newblock A curl preserving finite volume scheme by space velocity enrichment.
  application to the low mach number accuracy problem.
\newblock {\em Journal of Computational Physics}, 515:113252, 2024.

\bibitem{licht2017complexes}
Martin~Werner Licht.
\newblock Complexes of discrete distributional differential forms and their
  homology theory.
\newblock {\em Foundations of Computational Mathematics}, 17(4):1085--1122,
  2017.

\bibitem{milani2022artificial}
Riccardo Milani, J{\'e}r{\^o}me Bonelle, and Alexandre Ern.
\newblock Artificial compressibility methods for the incompressible
  {N}avier--{S}tokes equations using lowest-order face-based schemes on
  polytopal meshes.
\newblock {\em Computational Methods in Applied Mathematics}, 22(1):133--154,
  2022.

\bibitem{perrier2024developmentdiscontinuousgalerkinmethods}
Vincent Perrier.
\newblock Development of discontinuous galerkin methods for hyperbolic systems
  that preserve a curl or a divergence constraint, 2024.

\bibitem{raviart1977mixed}
Pierre-Arnaud Raviart and Jean-Marie Thomas.
\newblock A mixed finite element method for 2-nd order elliptic problems.
\newblock In {\em Mathematical aspects of finite element methods}, pages
  292--315. Springer, 1977.

\bibitem{raviart1977primal}
Pierre-Arnaud Raviart and Jean-Marie Thomas.
\newblock Primal hybrid finite element methods for 2nd order elliptic
  equations.
\newblock {\em Mathematics of Computation}, 31(138):391--413, 1977.

\bibitem{tavelli2017pressure}
Maurizio Tavelli and Michael Dumbser.
\newblock A pressure-based semi-implicit space--time discontinuous {G}alerkin
  method on staggered unstructured meshes for the solution of the compressible
  {N}avier--{S}tokes equations at all {M}ach numbers.
\newblock {\em Journal of Computational Physics}, 341:341--376, 2017.

\end{thebibliography}

\appendix

\section{Proof for low order quads}
\label{app:Quad}

\begin{prop}
  \label{quad_lowOrder}        
  Suppose that a periodic Cartesian mesh is composed of $N$ cells,
  and that on each
  cell (of mid point $(m_{i,j} ^x, m_{i,j} ^y)$ and of length $L_{i,j} ^x$ and
  $L_{i,j} ^y$),
  a vector $\bu$ is in
  $$\mathrm{Vec} \left(
  \left(
  \begin{array}{c}
    1 \\
    0
  \end{array}
  \right)
  ,
  \left(
  \begin{array}{c}
    0 \\
    1
  \end{array}
  \right)
  ,
  \left(
  \begin{array}{c}
    \dfrac{2}{L_{i,j} ^x} \, \left( m_{i,j} ^x - x \right) \\
    \dfrac{2}{L_{i,j} ^y} \, \left( y - m_{i,j} ^y \right)
  \end{array}
  \right)
  \right),$$
  then the vectorial space such that
  \begin{equation}
    \label{eq:udotNQuad}
    \forall f \qquad \Jump{\bu \cdot \bn_f} = 0,
  \end{equation}
  is of dimension $N+1$. 
\end{prop}

\begin{proof}
  We denote by $N_x$ (resp. $N_y$) the number of cells in the $x$
  (resp. $y$) direction. For each cell $i,j$, we denote by $\alpha_{i,j}$,
  $\beta_{i,j}$ and  $\gamma_{i,j}$ the coefficients such that
  $$\bu_{|c_{i,j}}
  =
  \alpha_{i,j}
  \left(
  \begin{array}{c}
    1 \\
    0
  \end{array}
  \right)
  +
  \beta_{i,j}
  \left(
  \begin{array}{c}
    0 \\
    1
  \end{array}
  \right)
  + \gamma_{i,j}
  \left(
  \begin{array}{c}
    \dfrac{2}{L_{i,j} ^x} \, \left( m_{i,j} ^x - x \right) \\
    \dfrac{2}{L_{i,j} ^y} \, \left( y - m_{i,j} ^y \right)
  \end{array}
  \right). 
  $$
  Then the constraint $\Jump{\bu \cdot \bn_f} = 0$ is equivalent to the
  following equations
  \begin{equation}
    \label{eq:j}
    \forall j \qquad 
    \left\{
    \begin{array}{l}
      \forall\,  0 \leq i \leq N_x-2 \qquad 
      \alpha_{i+1,j} -  \alpha_{i,j} =
      - \left( \gamma_{i+1,j} + \gamma_{i,j}\right)\\
      \alpha_{0,j} -  \alpha_{N_x-1,j} =
      - \left( \gamma_{0,j} + \gamma_{N_x-1,j}\right)\\
    \end{array}
    \right.
  \end{equation}
  \begin{equation}
    \label{eq:i}
    \forall i \qquad 
    \left\{
    \begin{array}{l}
      \forall\,  0 \leq j \leq N_y-2 \qquad 
      \beta_{i,j+1} -  \beta_{i,j} =
      - \left( \gamma_{i,j+1} + \gamma_{i,j}\right)\\
      \beta_{i,0} -  \beta_{i,N_y-1} =
      - \left( \gamma_{i,0} + \gamma_{i,N_y-1}\right)\\
    \end{array}
    \right.
  \end{equation}
  We first consider equation \eqref{eq:j}. If this equation is seen for each
  $j$ as a
  system in $\alpha_{\star,j}$, then this system is singular, with a kernel
  directed by $(1, 1,\dots,1)$, and the right hand side is compatible
  with this kernel if and only if
  \begin{equation}
    \label{eq:gammai}
    \forall j \qquad \Sum{i}{}\gamma_{i,j} = 0.
  \end{equation}
  If this last constraint holds, then if one $\alpha_{\star,j}$ is known
  (for example $\alpha_{0,j}$), then the $\alpha_{i,j}$ are known for all
  $i$. This makes $N_y$ parameters for recovering the $\alpha_{i,j}$ once the
  $\gamma_{i,j}$ are known. 

  In the same manner, by considering \eqref{eq:i}, we can prove that if and
  only if
  \begin{equation}
    \label{eq:gammaj}
    \forall i \qquad \Sum{j}{}\gamma_{i,j} = 0,
  \end{equation}
  the system in $\beta$ has a $N_x$ parameter family solution determined by
  the $\beta_{i,0}$.

  It remains to study the system on the $\gamma$ coefficients defined by
  \eqref{eq:gammai},\eqref{eq:gammaj}. For this system, we consider the
  coefficients $\gamma_{i,j}$ for $i>0$ and for $j>0$ as parameters
  (this makes $N-(N_x+N_y-1) = N -N_x - N_y +1$ parameters). Then
  the $\gamma_{0,j}$ for $j \geq 1$ are determined by \eqref{eq:gammai}, and
  the $\gamma_{i,0}$ for $i \geq 1$ are determined by \eqref{eq:gammaj}:
  \begin{equation}
    \label{eq:gammaEncore}
    \begin{array}{r@{\, = \, }l}
      \forall j \geq 1 \qquad \gamma_{0,j} &  - \Sum{i \geq 1}{} \gamma_{i,j}\\
      \forall i \geq 1 \qquad \gamma_{i,0} &  - \Sum{j \geq 1}{} \gamma_{i,j}\\
    \end{array}
  \end{equation}
  It remains to determine $\gamma_{0,0}$, which is a priori given by
  two equations
  $$
  \begin{array}{r@{\, = \, }l}
    \gamma_{0,0} & - \Sum{i \geq 1}{} \gamma_{i,0}\\
    \gamma_{0,0} & - \Sum{j \geq 1}{} \gamma_{0,j}\\
  \end{array}
  $$
  However, if \eqref{eq:gammaEncore} is considered, the two equations give the
  same value, namely
  $$  \gamma_{0,0} = - \Sum{i \geq 1}{} \Sum{j \geq 1}{} \gamma_{i,j}.$$
  We therefore have been able to express all the $\gamma_{0,j}$ and all the
  $\gamma_{i,0}$ provided all the $\gamma_{i,j}$ for $i \geq 1$ and
  $j \geq 1$ are known.
 
  Finally, a basis of the vectorial space determined
  by \eqref{eq:udotNQuad} was found. Its parameters are
  \begin{itemize}
  \item the $\gamma_{i,j}$ for $i \geq 1$ and $j \geq 1$. These are
    $N -N_x - N_y +1$ parameters.
  \item the $\alpha_{0,j}$ for $j \geq 0$. These are
    $N_y$ parameters.
  \item the $\beta_{i,0}$ for $i \geq 0$. These are
    $N_x$ parameters.
  \end{itemize}
  This makes a total of $N+1$ parameters. 

\end{proof}

\end{document}